\documentclass{amsart}
\calclayout
\usepackage{amssymb,amsmath,amsfonts,epsfig,latexsym,tikz}
\usepackage[alphabetic, initials]{amsrefs}
\usepackage{tikz-cd}
\usepackage{hyperref}
\usepackage{enumerate}
\usepackage{mathtools}
\usepackage{verbatim}
\usepackage{cleveref}
\usepackage[shortlabels]{enumitem}
\usepackage{subcaption}
\usepackage{ mathrsfs }

\usetikzlibrary{positioning}
\usetikzlibrary{matrix}
\usetikzlibrary{decorations}
\usetikzlibrary{decorations.pathreplacing, decorations.pathmorphing, angles,quotes}

\newtheorem{theorem}{Theorem}[section]

\newtheorem{proposition}[theorem]{Proposition}
\newtheorem{lemma}[theorem]{Lemma}
\newtheorem{corollary}[theorem]{Corollary}
\newtheorem{conjecture}[theorem]{Conjecture}

\def\N{\mathbb{N}}

\def\R{\mathbb{R}}

\def\Z{\mathbb{Z}}
\def\K{\mathcal{K}}

\def\1{\mathbf{1}}

\def\<{\langle}
\def\>{\rangle}

\DeclareMathOperator{\sgn}{sgn}

\DeclareMathOperator{\Supp}{Supp}

\DeclareMathOperator{\ann}{ann}

\DeclareMathOperator{\HR}{HR}

\newcommand{\Conv}[1]{\operatorname{Conv}\left\{{#1}\right\}}

\theoremstyle{definition}
\newtheorem{remark}[theorem]{Remark}
\newtheorem{example}[theorem]{Example}

\begin{document}
	
	\title[]{Positivity properties of Schur classes}
	\date{\today}
	\author{Matt Larson and Alan Stapledon}

\address{Princeton University and the Institute for Advanced Study}
\email{mattlarson@princeton.edu}

\address{Sydney Mathematics Research Institute}

	\email{astapldn@gmail.com}
\begin{abstract}
Results of Fulton--Lazarsfeld and Ross--Toma show that Schur classes of ample and nef vector bundles have remarkable positivity properties. We generalize these results to a purely algebraic setting. We show that if a projective bundle ring, a ring which resembles the cohomology ring of a projective bundle over a smooth complex projective variety, has the K\"{a}hler package with respect to a suitable cone, then the associated Schur classes satisfy a version of the Hodge--Riemann relations. This result is new even for the projectivization of an ample vector bundle over a smooth complex projective variety. We apply this result to prove that Schur coefficients of matroids are nonnegative. 
\end{abstract}
	
	\maketitle
	
	\vspace{-20 pt}

	\section{Introduction}

In the 1960s, a wide variety of positivity properties of vector bundles on algebraic varieties were studied by Grauert, Griffiths, Hartshorne, and others. The most commonly studied notion was introduced by Hartshorne \cite{HartshorneAmple}. 
A vector bundle $E$ on a complex projective variety $X$ is \emph{ample} if the relative $\mathcal{O}(1)$ on the projectivization $\mathbb{P}_X({E})$ is an ample line bundle. A vector bundle is \emph{nef} if the relative $\mathcal{O}(1)$ is a nef line bundle. There is now a large collection of results showing that ample and nef vector bundles have properties analogous to those enjoyed by ample and nef line bundles; see \cite[Part Two]{Lazarsfeld2} for a comprehensive summary of the theory as of 2004.

Motivated by a problem of Griffiths \cite{Griffiths}, positivity properties for characteristic classes of ample or nef vector bundles have attracted significant attention. 
Bloch and Gieseker  showed that if the rank of a nef vector bundle $E$ is at least the dimension of $X$, then the degree of the top Chern class $c_{\dim X}(E)$ is nonnegative \cite{BlochGieseker}. This was earlier proved when $X$ is a surface by Kleiman \cite{Kleiman}. After some further work \cite{Giseker,UsuiTango}, this result was generalized to \emph{Schur classes} by Fulton and Lazarsfeld. Given a partition $\lambda = (\lambda_1, \dotsc, \lambda_e)$ with $\lambda_1 \ge \dotsb \ge \lambda_e > 0$, let $\ell(\lambda) = e$ be the number of parts and let $|\lambda| = \lambda_1 + \dotsb + \lambda_e$ be the sum of the parts. The Schur class $s_{\lambda}(E)$ is defined by the formula
\begin{equation}\label{eq:schurdef}
s_{\lambda}(E) = \det(s_{\lambda_i + j -i}(E))_{i,j=1}^{\ell(\lambda)}
\end{equation}
where $s_i(E) = s_{(i)}(E)$ is the $i$th Segre class of $E$. This is a class in $H^{2|\lambda|}(X)$, and it can also be defined as an explicit polynomial in Chern classes; see Section~\ref{sec:projbund}. In \cite{FultonLaz}, Fulton and Lazarsfeld showed that if $|\lambda| = \dim X$ and $E$ is nef, then the degree of $s_{\lambda}(E)$ is nonnegative. 
As $c_i(E) = s_{(1^i)}(E)$, this is a generalization of the result of Bloch and Gieseker. Furthermore, this result is sharp: any characteristic class of $E$ in the cohomology of $X$ can be expressed as a linear combination of Schur classes, and Fulton and Lazarsfeld constructed examples to show that $\sum_{\lambda} a_{\lambda} s_{\lambda}(E)$ has nonnegative degree for all nef vector bundles $E$ on all complex projective varieties $X$ if and only if $a_{\lambda} \ge 0$ for all $\lambda$. 

This result was extended to the setting of compact K\"{a}hler manifolds in \cite{DPS}. In a series of works, Ross and Toma have shown positivity properties of Schur classes on compact K\"{a}hler manifolds when $|\lambda| = \dim X - 2$ \cite{RT1,RT2,RT3}. 
These results were extended in \cite{HRLuZheng}. 

\medskip

We prove a generalization of these results on Schur classes in a purely algebraic setting. Beyond giving a new proof of the aforementioned results, our proof gives additional inequalities satisfied by Schur classes of nef vector bundles. These results are new even for nef vector bundles on smooth complex projective varieties. 

Let $A = A^0 \oplus \dotsb \oplus A^n$ be a commutative graded $\mathbb{R}$-algebra which is equipped with a linear map $\deg_A \colon A^n \to \mathbb{R}$. If $A$ is equipped with a nonempty open convex cone $\mathcal{K}_A \subseteq A^1$, then we say that $A$ has the \emph{K\"{a}hler package} with respect to $\mathcal{K}_A$ if it satisfies the following properties:
\begin{enumerate}
\item For all nonnegative $k$, the pairing
$$A^k \times A^{n-k} \to \mathbb{R}, \quad (x, y) \mapsto \deg_A(xy)$$
is nondegenerate (\emph{Poincar\'{e} duality}). 
\item For all nonnegative $k \le n/2$ and $h \in \mathcal{K}_A$, the map
$$A^k \to A^{n-k}, \quad x \mapsto h^{n - 2k} x$$
is an isomorphism (\emph{hard Lefschetz property}).
\item For all nonnegative $k \le n/2$ and $h \in \mathcal{K}_A$, the bilinear form
$$A^k \times A^k \to \mathbb{R}, \quad (x, y) \mapsto (-1)^k \deg_A(h^{n - 2k} xy)$$
is positive definite on the kernel of multiplication by $h^{n - 2k + 1}$.  
\end{enumerate}

For example, if $X$ is a compact K\"{a}hler manifold (such as a smooth complex projective variety), then the ring of real $(p, p)$ classes on $X$ satisfies the K\"{a}hler package with respect to the K\"{a}hler cone; see \cite[Chapter 3]{Huybrechts}. There are a number of combinatorial sources of algebras with the K\"{a}hler package \cite{McMullen1989,AHK18,ADH,LarsonPartida}. See \cite{HuhICM1} for a discussion. 

We will be concerned with rings that resemble the cohomology ring of a projective bundle. Recall that if $E$ is a rank $r$ vector bundle over $X$, then $H^*(\mathbb{P}_X(E)) \cong H^*(X)[\zeta]/(\zeta^r - c_1(E) \zeta^{r-1} + \dotsb + (-1)^r c_r(E))$, where $\zeta$ is the first Chern class of the relative $\mathcal{O}(1)$.  Given classes $c_1, \dotsc, c_r$, with $c_i \in A^i$, the \emph{projective bundle ring} $B$ is the ring defined by
$$B \coloneqq \frac{A[\zeta]}{(\zeta^r - c_1 \zeta^{r-1} + \dotsb + (-1)^r c_r)}.$$
Set $c_0 = 1$, and set $c_i = 0$ for $i > r$ or $i < 0$. 
Note that $B$ is a commutative graded $\mathbb{R}$-algebra, and there is a natural ring injection $\pi^* \colon A \to B$. The ring $B$ admits a direct sum decomposition
\begin{equation}\label{eq:Bdirectsum}
B = \bigoplus_{i=0}^{r-1} \zeta^i \cdot A.
\end{equation}
In particular, $B^{n + r - 1}$ is isomorphic to $\zeta^{r-1} A^n$, so there is an induced degree map $\deg_B \colon B^{n + r - 1} \to \mathbb{R}$. 
For any partition $\lambda$, there is a corresponding Schur class $s_\lambda \in A^{|\lambda|}$ (see Section~\ref{sec:projbund}). For example, if $\lambda = (1^i)$ then $s_\lambda = c_i$. 
Let $\mathcal{K}_B$ be the open convex cone in $B^1$ given by $\pi^*(\K_A) + \R_{> 0} \zeta$. 
 It turns out that, for any choice of $c_1, \dotsc, c_r$, $B$ has the K\"{a}hler package with respect to some cone; see Proposition~\ref{prop:somecone}. We will show that if $B$ has the K\"{a}hler package with respect to $\mathcal{K}_B$, then the classes $c_1, \dotsc, c_r$ have remarkable positivity properties. 

\medskip

To state our main result, we will need to make some combinatorial definitions. 
We often identify partitions with their corresponding Young diagrams (with row indices increasing downwards starting from $1$ and column indices increasing to the right starting from $1$). In particular, we write $\lambda \subseteq \mu$ for a partition $\mu$ if there is a containment of the corresponding Young diagrams. 
We write $\lambda + \mu$ for component-wise addition of partitions. 
Fix a partition $\lambda$, and let its transpose be the partition
\[
(m_1,\ldots,m_1,m_2,\ldots,m_2,\ldots,m_t,\ldots,m_t),
\]
where $m_1 > m_2 > \cdots > m_t$ and $m_i$ appears $k_i$ times. For
$1 \le i \le t$, let $\widehat{\lambda}_i$ be the rectangular partition
\[
(\underbrace{k_i,\ldots,k_i}_{m_i\text{ times}}).
\]
Then $\ell(\widehat{\lambda}_i) = m_i$ and $\lambda=\widehat{\lambda}_1+\cdots+\widehat{\lambda}_t.$
We call this the \emph{rectangular decomposition} of $\lambda$.
Observe that the distinct parts of $\lambda$ are $\{ \lambda_{m_i} = k_1 + \cdots + k_i : 1 \le i \le t \}$.  Set $m_{t + 1} = 0$ and $\lambda_{0} = \lambda_{m_{t + 1}} = \infty$. 
For $1 \le i \le t$, define the skew Young diagram
\begin{equation}\label{eq:rectangledef}
R_i = \{ (r,c) : m_{i + 1} < r \le m_i,\, \lambda_{m_i} < c \le \lambda_{m_{i + 1}} \}. 
\end{equation}
Observe that $R_1,\ldots,R_{t - 1}$ are finite rectangles, while $R_t$ has infinitely many columns. See Figure~\ref{fig:rectangular-decomposition} for an illustration when $t = 3$. 

\begin{figure}[htbp]
	\centering
	\begin{tikzpicture}[x=0.9cm,y=0.9cm, every node/.style={font=\Large}]
		
		\definecolor{lambdaBlue}{RGB}{215,232,255}
		\definecolor{regionYellow}{RGB}{255,238,175}
		
		\fill[lambdaBlue] (0,0) rectangle (2,6);
		\fill[lambdaBlue] (2,2) rectangle (4,6);
		\fill[lambdaBlue] (4,4) rectangle (6,6);
		
		\fill[regionYellow] (2,0) rectangle (4,2);
		\fill[regionYellow] (4,2) rectangle (6,4);
		\fill[regionYellow] (6,4) rectangle (11,6);
		
		\draw[thin] (2,0) rectangle (4,2);
		\draw[thin] (4,2) rectangle (6,4);
		\draw[thin] (6,4) -- (11,4);
		\draw[thin] (6,6) -- (11,6);
		\draw[thin] (6,4) -- (6,6);
		
		\draw[thin] (2,2) -- (2,6);
		\draw[thin] (4,4) -- (4,6);
		
		\draw[line width=1.3pt]
		(0,0) -- (0,6) -- (6,6) -- (6,4) -- (4,4) -- (4,2) -- (2,2) -- (2,0) -- cycle;
		
		\node at (1,3) {$\widehat{\lambda}_1$};
		\node at (3,4) {$\widehat{\lambda}_2$};
		\node at (5,5) {$\widehat{\lambda}_3$};
		
		\node at (3,1) {$R_1$};
		\node at (5,3) {$R_2$};
		\node at (8.5,5) {$R_3$};
		\node at (10.2,5) {$\cdots$};

	\end{tikzpicture}
	
		\caption{The rectangular decomposition
		\(\lambda=\widehat{\lambda}_1+ \cdots +\widehat{\lambda}_t\)
		and the associated rectangles \(R_1,\ldots,R_t\) when \(t = 3\).}
	\label{fig:rectangular-decomposition}

\end{figure}
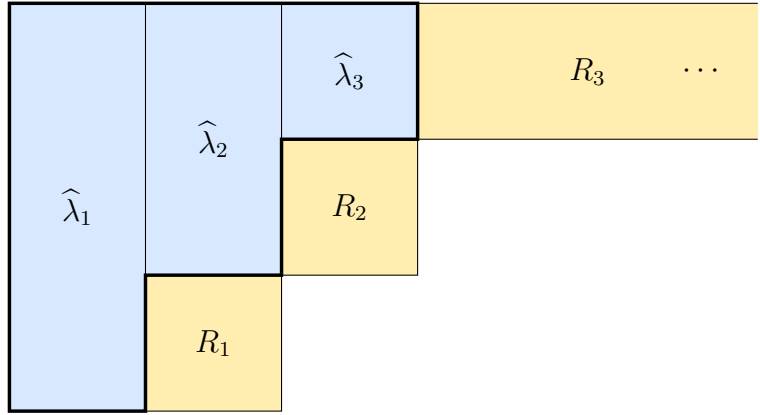

\begin{theorem}\label{thm:main}
Let $A = A^0 \oplus \dotsb \oplus A^n$ be an algebra with the K\"{a}hler package with respect to a cone $\mathcal{K}_A$. For some $r \ge 1$, choose classes $c_i \in A^i$ for $i = 1, \dotsc, r$. Let $B = A[\zeta]/(\zeta^r - c_1 \zeta^{r-1} + \dotsb + (-1)^r c_r)$. Suppose that $B$ has the K\"{a}hler package with respect to $\mathcal{K}_B$. 
Let $k \le n/2$ be a nonnegative integer. 
Let $\lambda$ be a partition with $|\lambda| = n - 2k$, and 
consider an element $a \in A^k$. Fix $h \in \K_A$ and consider  the rectangular decomposition of $\lambda$ with corresponding rectangles $R_1,\ldots,R_t$ as defined in \eqref{eq:rectangledef}.
Suppose that for any partition $\mu$ with $\lambda \subseteq \mu$ and $\ell(\lambda) = \ell(\mu)$, either
 \begin{enumerate}
 	
 	\item\label{i:primitiveinside} $\mu \subseteq \lambda \cup R_1 \cup \cdots \cup R_t$  and 
 	\begin{equation}\label{eq:primitive}
 		 	a  \left( h s_\mu +   \sum_{i = 1}^t i \sum_{ \substack{ \nu \supseteq \mu \\ |\nu| = |\mu| + 1 \\ \nu \smallsetminus \mu \in R_i}} s_\nu \right) = 0, \textrm{ or},
 	\end{equation}

 	\item\label{i:primitiveoutside} $\mu \nsubseteq \lambda \cup R_1 \cup \cdots \cup R_t$ and  $as_\mu = 0$. 
 \end{enumerate}
 Then 
 \[
 (-1)^k\deg_A(a^2 s_\lambda) \ge 0,
 \]
 with equality if and only if $a s_\mu = 0$ in $A$ for all partitions $\mu$ with $\lambda \subseteq \mu$. 
\end{theorem}

We make some comments on the two conditions appearing in Theorem~\ref{thm:main}. 

\begin{itemize}
	\item Condition \eqref{i:primitiveinside} holds if $|\mu| \ge |\lambda| + k$. Indeed, 
	then all terms of the form $ahs_\mu$ and $as_\nu$ in Equation~\eqref{eq:primitive} have degree strictly greater than $n$ and so are zero. Similarly, \eqref{i:primitiveoutside} holds if $|\mu| > |\lambda| + k$.
	\item 	The condition $\mu \nsubseteq \lambda \cup R_1 \cup \cdots \cup R_t$ implies that $|\mu| \ge |\lambda| + 3$. Indeed, given a box $b$ in $\mu \smallsetminus (\lambda \cup R_1 \cup \cdots \cup R_t)$, the boxes immediately above and to the left of $b$ lie in $\mu \smallsetminus \lambda$. In particular, condition \eqref{i:primitiveoutside} holds vacuously if $k \le 2$. 
\end{itemize}

When $A$ is the ring of real $(p,p)$ classes on a compact K\"{a}hler manifold and $B$ is the ring of real $(p,p)$ classes on the projectivization of a nef vector bundle, then $B$ has the K\"{a}hler package with respect to $\mathcal{K}_B$ and the hypothesis of Theorem~\ref{thm:main} holds. 

\begin{example}\label{ex:FL}
Suppose that $k=0$, so $|\lambda| = n$. Then \eqref{i:primitiveinside} and \eqref{i:primitiveoutside} are automatic because $A^s = 0$ for $s > n$. So if $B$ has the K\"{a}hler package with respect to $\mathcal{K}_B$, then $\deg_A(s_{\lambda}) \ge 0$. 
\end{example}

In particular, Example~\ref{ex:FL} generalizes the main result of \cite{FultonLaz} for smooth complex projective varieties. Using intersection cohomology, our proof can also be adapted to arbitrary complex projective varieties; see Theorem~\ref{thm:moduleversion}. Our arguments cannot be easily adapted to varieties over a field of positive characteristic. 

\begin{example}\label{ex:rossToma}
Suppose that $k = 1$, so $|\lambda| = n-2$. Then \eqref{i:primitiveoutside} is vacuous and \eqref{i:primitiveinside} is vacuous unless $\mu = \lambda$. We see that the hypothesis of Theorem~\ref{thm:main} holds for $a \in A^1$ which lie inside a subspace of codimension at most $1$. This implies that the bilinear form $A^1 \times A^1 \to \mathbb{R}$ given by $(x, y) \mapsto \deg_A(xy s_{\lambda} )$ has at most one positive eigenvalue. 
\end{example}

Example~\ref{ex:rossToma} is a generalization of the main theorem of \cite{RT2}. 
If $k > 1$, then Theorem~\ref{thm:main} gives a version of the higher Hodge--Riemann relations for Schur classes of nef vector bundles. This answers \cite[Question 9.1]{RT1}.

Theorem~\ref{thm:main} gives a generalization of known positivity properties of nef vector bundles. We now state a version which gives a stronger positivity result under an assumption that holds for Chern classes of ample vector bundles. 

\begin{theorem}\label{thm:strict}
Let $A = A^0 \oplus \dotsb \oplus A^n$ be an algebra with the K\"{a}hler package with respect to a cone $\mathcal{K}_A$. For some $r \ge 1$, choose classes $c_i \in A^i$ for $i = 1, \dotsc, r$. Let $B = A[\zeta]/(\zeta^r - c_1 \zeta^{r-1} + \dotsb + (-1)^r c_r)$. Suppose that $B$ has the K\"{a}hler package with respect to $\mathcal{K}_B$, and that $\zeta$ satisfies the hard Lefschetz property in $B$. Then for any partition $\lambda$ with $|\lambda| = n$ and at most $r$ parts, we have $\deg_A(s_{\lambda}) > 0$. 
\end{theorem}

This recovers \cite[Theorem 1]{FultonLaz}. Theorem~\ref{thm:strict} is deduced from Theorem~\ref{thm:main} via a perturbation argument. 

\medskip

We discuss two ingredients in the proof of Theorem~\ref{thm:main}. We retain the notation of the theorem. If $E$ is a vector bundle over a variety $X$, then we can consider the multiprojective bundle $\mathbb{P}_X(E) \times_{X} \mathbb{P}_X(E) \times_X \dotsb \times_X \mathbb{P}_X(E)$ obtained by taking the fiber product of $m$ copies of $\mathbb{P}_X(E)$ over $X$. We will consider rings which resemble the cohomology rings of multiprojective bundles. For some $m \ge 1$, let 
$$C := \frac{A[\zeta_1, \dotsc, \zeta_m]}{(\zeta_1^{r} - c_1 \zeta_1^{r-1} + \dotsb + (-1)^r c_r, \dotsc, \zeta_m^r - c_1 \zeta_{m}^{r-1} + \dotsb + (-1)^r c_r)}.$$
There is a natural inclusion $p^* \colon A \to C$ which identifies $A$ with a subalgebra of $C$. 
Then $C$ admits a direct sum decomposition analogous to \eqref{eq:Bdirectsum}, and there is a distinguished degree map $\deg_C \colon C^{n + m(r-1)} \to \mathbb{R}$. Let $\mathcal{K}_C$ be the open convex cone in $C^1$ given by $p^*(\mathcal{K}_A) + \mathbb{R}_{>0} \zeta_1 + \dotsb + \mathbb{R}_{>0} \zeta_m$. 

\begin{theorem}\label{thm:multiproj}
Suppose that $B$ satisfies the K\"{a}hler package with respect to $\mathcal{K}_B$. Then $C$ has the K\"{a}hler package with respect to $\mathcal{K}_C$. 
\end{theorem}

The proof of Theorem~\ref{thm:multiproj} uses a deep result of Kashiwara and Kawai \cite{KK87}. The other ingredient is a new formula for Schur classes. In $C$, let $\Delta = \prod_{1 \le i < j \le m} (\zeta_i - \zeta_j)$. If $m=1$, then we have $\Delta = 1$. There is a pushforward map $p_* \colon C \to A$ which is adjoint to $p^*$ under the symmetric bilinear forms induced by Poincar\'{e} duality. If $A$ is the cohomology ring of a smooth complex projective variety and $C$ is the cohomology ring of a multiprojective bundle, then $p_*$ is the Gysin pushforward map. For a partition $\lambda$, let $s_{\lambda}(\zeta_1, \dotsc, \zeta_m)$ be the Schur polynomial of $\lambda$ evaluated at $\zeta_1, \dotsc, \zeta_m$. This is a class in $C^{|\lambda|}$, and it is the Schur class that we obtain if we use the elementary symmetric functions of $\zeta_1, \dotsc, \zeta_m$ as the $c_i$. 

\begin{proposition}\label{prop:pushforward}
Let $\lambda$ be a partition, and choose some $m$ with $\ell(\lambda) \le m \le r$. Then
$$(-1)^{\binom{m}{2}} m! s_{\lambda} = p_* \left(\Delta^2 (\zeta_1 \dotsb \zeta_m)^{r-m} s_{\lambda}(\zeta_1, \dotsc, \zeta_m)\right ).$$
\end{proposition}

We now sketch the proof of Theorem~\ref{thm:main}. Fix a partition $\lambda$ with size $n - 2k$ and a class $a \in A^k$ satisfying the hypothesis of Theorem~\ref{thm:main}. We induct on the number of parts in the rectangular decomposition of $\lambda$.
The base case is when $\lambda$ is the empty partition,
when Theorem~\ref{thm:main} reduces to the Hodge--Riemann relations for $A$ in degree $n/2$. To do the inductive step, we consider the multiprojective bundle ring $C$ given by taking $m = \ell(\lambda)$. Let $\overline{\lambda}$ be the partition obtained by removing the first block of the rectangular decomposition from $\lambda$. If $k_1$ is the number of columns of $\lambda$ with length $m$, then  $s_{\lambda}(\zeta_1, \dotsc, \zeta_m) = (\zeta_1 \dotsb \zeta_m)^{k_1} s_{\overline{\lambda}}(\zeta_1, \dotsc, \zeta_m)$. Proposition~\ref{prop:pushforward} then gives the formula
$$(-1)^{\binom{m}{2}} m! \deg_A(a^2s_{\lambda}) = \deg_C(a^2\Delta^2 (\zeta_1 \dotsb \zeta_m)^{r - m + k_1} s_{\overline{\lambda}}(\zeta_1, \dotsc, \zeta_m)).$$
Let $D \coloneqq C/\operatorname{ann}((\zeta_1 \dotsb \zeta_m)^{r - m + k_1})$ be the quotient of $C$ by the annihilator of $(\zeta_1 \dotsb \zeta_m)^{r - m + k_1}$; this algebra satisfies Poincar\'{e} duality with respect to a natural degree map. We have 
$$\deg_C(a^2\Delta^2 (\zeta_1 \dotsb \zeta_m)^{r - m + k_1} s_{\overline{\lambda}}(\zeta_1, \dotsc, \zeta_m)) = \deg_{D}(a^2 \Delta^2 s_{\overline{\lambda}}(\zeta_1, \dotsc, \zeta_m)).$$
By the descent lemma of \cite{CKS87} and \cite{KK87} and Theorem~\ref{thm:multiproj}, $D$ has the K\"{a}hler package with respect to the image of $\mathcal{K}_C$. It also follows from Theorem~\ref{thm:multiproj} that $D[\xi]/((\xi - \zeta_1) \dotsb (\xi - \zeta_m))$ satisfies the K\"{a}hler package with respect to the expected cone; see Proposition~\ref{prop:nefdirectsum}. That $a$ satisfies the hypothesis of Theorem~\ref{thm:main} in $A$ implies that $a \Delta$ satisfies the hypothesis of Theorem~\ref{thm:main} in $D$, and the result follows from induction.

The proof in \cite{FultonLaz} reduces the positivity of Schur classes of ample vector bundles to the Bloch--Gieseker theorem \cite{BlochGieseker}.
In \cite[pg. 44]{FultonLaz}, Fulton and Lazarsfeld ask for an alternative argument avoiding that. In our argument, we obtain the Bloch--Gieseker theorem as a special case, answering their question. 

\subsection{Applications to matroids}

We apply Theorem~\ref{thm:main} to prove a conjecture about matroids. We first describe the motivating geometry. Let $H_1, \dotsc, H_n$ be hyperplanes in an $r$-dimensional vector space $L$ over a field $k$. We assume that $\cap_i H_i = 0$. Choosing a linear form defining each hyperplane realizes $L$ as a subspace of $k^n$, where the $H_i$ are the restrictions of the coordinate hyperplanes. We have a corresponding point $[L]$ of the Grassmannian $\operatorname{Gr}(r,n)$. The torus $T = \mathbb{G}_m^n$ acts on $\operatorname{Gr}(r,n)$; the points of $T \cdot [L]$ correspond to different choices of defining linear forms. 

The Grassmannian $\operatorname{Gr}(r,n)$ has a stratification by Schubert varieties indexed by partitions $\lambda$ which fit inside an $r \times (n-r)$ box. The Schubert variety $\Omega_{\lambda}$ corresponding to $\lambda$ has \emph{dimension} equal to $|\lambda|$. Let $\mathcal{S}^{\vee}$ be the dual of the universal subbundle on $\operatorname{Gr}(r,n)$. In the homology of $\operatorname{Gr}(r,n)$, we have the identity
$$\deg(s_{\mu}(\mathcal{S}^{\vee}) \smallfrown [\Omega_{\lambda}]) = \begin{cases} 1 & \text{if }\lambda = \mu, \\ 0 & \text{otherwise}.\end{cases}$$
The classes $[\Omega_{\lambda}]$ form a basis for the homology of $\operatorname{Gr}(r,n)$; see \cite[Chapter 14]{FultonIntersection}.
We can expand the fundamental class of the torus-orbit closure $\overline{T \cdot [L]}$ in this basis:
$$[\overline{T \cdot [L]}] = \sum_{\lambda} a_{\lambda} [\Omega_{\lambda}].$$
We see that $a_{\lambda}$ is the degree of the Schur class $s_{\lambda}$ of the restriction of $\mathcal{S}^{\vee}$ to $\overline{T \cdot [L]}$. It follows, for example from \cite{FultonLaz}, that the $a_{\lambda}$ are nonnegative integers. 

The $a_{\lambda}$ are invariants of the hyperplane arrangement.  The \emph{matroid} of the hyperplane arrangement corresponding to $L \subseteq k^n$ is the collection of subsets $B$ of $\{1, \dotsc, n\}$ of size $r$ such that the composition
$$L \hookrightarrow k^n \twoheadrightarrow k^B$$
is an isomorphism, where the map $k^n \twoheadrightarrow k^B$ is the coordinate projection. 
In \cite{FinkSpeyer}, Fink and Speyer used equivariant localization to show that the $a_{\lambda}$ depend only on the matroid of $L$. Furthermore, they gave a (complicated) formula for the class of $\overline{T \cdot [L]}$ which makes sense for any matroid. This allows one to define an invariant which assigns a rank $r$ matroid $\mathrm{M}$ on $\{1, \dotsc, n\}$ to $a_{\lambda}(\mathrm{M})$, the coefficient of the expansion of Fink and Speyer's formula in the Schubert basis. This integer is called the \emph{Schur coefficient} of $\mathrm{M}$. 
However, from Fink and Speyer's formula it is no longer obvious that Schur coefficients are nonnegative. The nonnegativity of Schur coefficients was conjectured implicitly in \cite{Speyerg} and explicitly in \cite[Conjecture 9.13]{BergetFink}.

Schur coefficients of matroids have attracted significant attention. They were first studied by Klyachko \cite{KlyachkoOrbit}, who calculated the Schur coefficients of a uniform matroid, i.e., for hyperplanes in general position. 
An elegant geometric explanation of Klyachko's formula was given in \cite{Lian}. Klyachko's computation was generalized to sparse paving matroids by Hamre \cite{Hamre}, who proved the nonnegativity of Schur coefficients in this case. Schur coefficients for rank $2$ matroids were computed in \cite{EFG}, and Schur coefficients for snake matroids and lattice path matroids were computed in \cite{HSVV}. Perhaps the most attractive results on Schur coefficients are the following explicit formulas.

\begin{example}
In \cite[Theorem 5.1]{Speyerg}, Speyer showed that if $\lambda = (n-r, 1, \dotsc, 1)$ is the hook, then $a_{\lambda}(\mathrm{M})$ is Crapo's beta invariant of a matroid \cite{crapo}. This counts the number of bases of $\mathrm{M}$ according to certain statistics: it is the number of bases with external activity $0$ and internal activity $1$. 
\end{example}

\begin{example}
If $n \ge 2r+1$ and $\lambda = (n-r, r-1)$, then it follows from \cite[Corollary 6.16]{BF24} and Pieri's rule that $a_{\lambda}(\mathrm{M})$ is equal to the number of bases of the diagonal Dilworth truncation $D(\mathrm{M}, \mathrm{M})$ with generalized external activity $0$, a statistic introduced in \cite{BF24}. 
\end{example}

An alternative description of Schur coefficients of matroids was given in \cite{BEST}. Suppose that $L \subseteq k^n$ is a realization of a \emph{connected} matroid. This means that the stabilizer of $[L]$ in $T$ is as small as possible, i.e., $\dim T \cdot [L] = n-1$. Then there is a distinguished resolution of singularities $f \colon X_n \to \overline{T \cdot [L]}$ called the \emph{permutohedral toric variety}. By the projection formula, one has
$$a_{\lambda}(\mathrm{M}) = \deg(s_{\lambda}(f^* \mathcal{S}^{\vee}) \smallfrown [X_n]).$$
In \cite{BEST}, the authors gave a formula for the Chern classes of $f^* \mathcal{S}^{\vee}$ which depends only on the matroid of $L$. This formula makes sense for any matroid $\mathrm{M}$, and they used it to define classes $c_1(\mathcal{S}_{\mathrm{M}}^{\vee}), \dotsc, c_r(\mathcal{S}_{\mathrm{M}}^{\vee})$ in $H^*(X_n; \mathbb{R})$, the cohomology ring of the permutohedral toric variety. The ring $H^*(X_n; \mathbb{R})$ satisfies the K\"{a}hler package with respect to the ample cone $\mathcal{K}_{X_n}$ of the permutohedral toric variety. The following result was recently proved by Partida and the first author.

\begin{proposition}\cite[Theorem 1.4]{LarsonPartida}\label{prop:LP}
Let $\mathrm{M}$ be a matroid of rank $r$ on $\{1, \dotsc, n\}$. Then $H^*(X_n)[\zeta]/(\zeta^r - c_1(\mathcal{S}_{\mathrm{M}}^{\vee}) \zeta^{r-1} + \dotsb + (-1)^r c_r(\mathcal{S}_{\mathrm{M}}^{\vee}))$ satisfies the K\"{a}hler package with respect to the cone $\pi^*(\mathcal{K}_{X_n}) + \mathbb{R}_{>0} \zeta$. 
\end{proposition}

\begin{corollary}
Let $\mathrm{M}$ be a matroid of rank $r$ on $\{1, \dotsc, n\}$ and let $\lambda$ be a partition which fits inside the $r \times (n-r)$ box. Then $a_{\lambda}(\mathrm{M}) \ge 0$. 
\end{corollary}

\begin{proof}
Suppose that $\mathrm{M}$ is connected, so $a_{\lambda}(\mathrm{M}) = 0$ unless $|\lambda| = n-1$. Choose a partition $\lambda$ which fits inside the $r \times (n-r)$ rectangle with $|\lambda| = n-1$. If we define the Schur class $s_{\lambda} \in H^*(X_n; \mathbb{R})$ using $c_1(\mathcal{S}_{\mathrm{M}}^{\vee}), \dotsc, c_r(\mathcal{S}_{\mathrm{M}}^{\vee})$, then it is not hard to show that
$$a_{\lambda}(\mathrm{M}) = \deg(s_{\lambda});$$
see \cite[Remark 9.15]{BergetFink}. The result follows in this case from Theorem~\ref{thm:main} and Proposition~\ref{prop:LP}. The result can be reduced to the case of connected matroids \cite[Theorem 4.1]{Hamre}, and it is also not hard to directly prove the disconnected case by working on a product of permutohedral toric varieties. 
\end{proof}

The above argument and Theorem~\ref{thm:main} also answer \cite[Question 1.4]{BEST}, which asks for log-concavity properties satisfied by Schur classes. Furthermore, \cite[Theorem 1.4]{LarsonPartida} is considerably more general than Proposition~\ref{prop:LP}, and it can be used to answer analogous questions arising from \cite{EHL} and \cite{EFLS}. See \cite[Example 1.2 and Example 1.3]{LarsonPartida}.

In \cite{HuhICM1}, Huh asks for combinatorial applications of the higher Hodge--Riemann relations. As discussed above, the proof of Theorem~\ref{thm:main} crucially uses the higher Hodge--Riemann relations, even in the case of Fulton--Lazarsfeld positivity. The above argument therefore gives a combinatorial application of the higher Hodge--Riemann relations. 

\medskip

The paper is organized as follows. In Section~\ref{sec:symmetric}, we prove some preliminary results about symmetric functions which will be needed in the proof of Theorem~\ref{thm:main}. In Section~\ref{sec:projbund}, we develop some properties of projective bundle rings and prove Theorem~\ref{thm:multiproj} and Proposition~\ref{prop:pushforward}. In Section~\ref{sec:positivity}, we prove Theorems~\ref{thm:main} and ~\ref{thm:strict}. In Section~\ref{sec:variants}, we consider some variants of the main results.

\subsection*{Acknowledgements}

We thank Dave Anderson, Anders Buch, Teddy Gonzales, Rob Lazarsfeld, Julius Ross, Christian Schnell, and Hunter Spink for helpful conversations. Claude was used to help with numerical experiments and for proofreading, but all ideas in this paper were generated by the authors. 

\section{Symmetric functions}\label{sec:symmetric}

In this section, we develop some properties of symmetric functions that will be used in the proof of Theorem~\ref{thm:main}; see \cite[Chapter 7]{stanley-ec2} for a comprehensive reference.  
Let $\Lambda = \Lambda_\Z$ 
be the graded $\Z$-algebra of symmetric functions in infinitely many variables. Recall that  the Schur functions $\{ s_\lambda : \lambda \textrm{ a partition} \}$ form a $\Z$-basis of $\Lambda$, and 
multiplication in $\Lambda$ is determined by 
\[
s_\lambda s_\mu = \sum_{\nu} c_{\lambda,\mu}^\nu s_{\nu},
\]
where $\lambda, \mu$, and $\nu$ are partitions, and  $c_{\lambda,\mu}^\nu \in \Z_{\ge 0}$ is the corresponding Littlewood--Richardson coefficient. 
Here $s_\lambda$ is  homogeneous of degree $|\lambda|$.
Also, recall that $c_{\lambda,\mu}^\nu = 0$ unless $\lambda \subseteq \nu$ and $\mu \subseteq \nu$. 
For example, when $\lambda = (j)$ for some $j \in \Z_{\ge 0}$, we write $h_j = s_{\lambda}$ for the corresponding complete homogeneous symmetric function.
When $\lambda = (1^j)$ for  some $j \in \Z_{\ge 0}$, we write $e_j = s_{\lambda}$ for the corresponding elementary symmetric function. Recall that  $\Lambda$ is the polynomial ring on either $h_1, h_2, \dotsc$ or on $e_1, e_2, \dotsc$. 
We also recall Pieri's rule, that
$$h_j s_{\lambda} = \sum_{\mu} s_{\mu},$$
where the sum is over all partitions $\mu$ whose Young diagram is obtained by adding a horizontal $j$-strip to the Young diagram of $\lambda$, i.e., adding $j$ boxes, no two of which lie in the same column.

Considering length of partitions induces an increasing filtration of $\Lambda$ defined by
\[
\Lambda^{(\le m)} := \bigoplus_{\ell(\lambda) \le m } \Z \cdot s_\lambda,  
\]
for all $m \in \Z_{\ge 0}$. In particular, $\Lambda^{(\le 0)} = \Z \cdot s_{\emptyset} = \Z$.  
As an abelian group, 
$\Lambda^{(\le m)}$ is naturally isomorphic to the graded ring $\Z[x_1,\ldots,x_m]^{S_m}$ of symmetric functions in $m$ variables, with $s_\lambda \in \Lambda^{(\le m)}$ corresponding to $s_\lambda(x_1,\ldots,x_m) \in \Z[x_1,\ldots,x_m]^{S_m}$ for $\ell(\lambda) \le m$.  In particular, this isomorphism induces a natural ring structure on $\Lambda^{(\le m)}$. Explicitly, 
multiplication is given by the product
\[
\star_m \colon \Lambda^{(\le m)} \times \Lambda^{(\le m)} \to \Lambda^{(\le m)}, \quad
s_\lambda \star_m s_\mu = \sum_{\ell(\nu) \le m } c_{\lambda,\mu}^\nu s_{\nu}, 
\]
for any partitions $\lambda,\mu$ with $\ell(\lambda),\ell(\mu) \le m$. We have a surjective ring homomorphism 
$$\pi_m \colon  \Lambda \to \Lambda^{(\le m)}, \quad \pi_m(s_\lambda) = \begin{cases}
	s_\lambda &\textrm{if } \ell(\lambda) \le m, \\
	0         &\textrm{otherwise. }
\end{cases} 
$$
In what follows, we often consider $\star_m$ acting on elements of $\Lambda^{(\le m')}$ and $\Lambda^{(\le m'')}$  for $m', m'' \le m$ via the natural inclusions. 

\begin{example}\label{ex:rectangle}
	If $\lambda = (k,k,\ldots,k)$ is a rectangle with $\ell(\lambda) = m$, 
	then $s_\lambda(x_1,\dots,x_m) = (x_1\cdots x_m)^k$, and, for any partition $\mu$ with $\ell(\mu) \le m$, $s_\lambda \star_m s_\mu = s_{\lambda + \mu}$. 
\end{example}

Let $\lambda$ be a partition, and consider the ideal in $\Lambda$
\[
\Lambda_\lambda := \bigoplus_{  \mu \supseteq \lambda  } \Z \cdot s_{\mu} \subseteq\Lambda.
\]
For a nonnegative integer $m$, consider the ideal in $\Lambda^{(\le m)}$
\[
\Lambda_\lambda^{(\le m)} := \pi_m(\Lambda_\lambda) =  \bigoplus_{ \substack{ \mu \supseteq \lambda \\ \ell(\mu) \le m} } \Z \cdot s_{\mu} \subseteq \Lambda^{(\le m)}. 
\]
We can view $\Lambda_{\lambda}^{(\le m)}$ as a subset of $\Lambda$ or as a subset of $\Lambda_{\lambda}^{(\le m')}$ for any $m' \ge m$ via the natural inclusion. We need the following basic lemma.

\begin{lemma}\label{lem:idealgenerated}
	Let $\lambda$ be a partition with $\ell(\lambda) = m$. 
	Then the ideal 	$\Lambda_\lambda$ of $\Lambda$ is generated by $\Lambda_\lambda^{(\le m)}$. 
	In particular, for any $m' \ge m$, the ideal $\Lambda_\lambda^{(\le m')}$ in   $\Lambda^{(\le m')}$ is generated by $\Lambda_\lambda^{(\le m)}$.
\end{lemma}
\begin{proof}
	The second statement follows from the first by applying $\pi_{m'}$. 
	For the first statement, let $I$ be the ideal in $\Lambda$ generated by 
	$\Lambda_\lambda^{(\le m)}$; clearly
	$I \subseteq \Lambda_\lambda$. We prove the reverse inclusion by induction,  firstly on the length of a partition, and then on the number of elements in its last part.
	Consider a partition $\mu$ with $\mu \supseteq \lambda$. Let $j = \ell(\mu) \ge  m$, and write  $\mu = (\mu_1,\ldots,\mu_j)$. 
	If $j = m$, then $s_\mu \in I$ by definition. Assume that $j > m$, and let $\mu' = (\mu_1,\ldots,\mu_{j - 1})$.
	Then $\lambda \subseteq \mu'$, and hence $s_{\mu'} \in I$ by induction on the length of a partition. 
	By Pieri's rule, $h_{\mu_j}s_{\mu'} \in I$ is the sum of Schur functions whose partitions are obtained from $\mu'$ by adding a horizontal $\mu_j$-strip to its Young diagram. Hence  $h_{\mu_j}s_{\mu'} = s_{\mu} + \sum_{\nu} s_{\nu}$, where each $\nu = (\nu_1,\ldots,\nu_j)$ satisfies $0 \le \nu_j < \mu_j$. By induction, each $s_{\nu}$ is contained in $I$, and so $s_{\mu}$ is contained in $I$. 
\end{proof}

Consider a  partition $\lambda$ with $\ell(\lambda) = m$. Recall that it has a rectangular decomposition 
$\lambda=\widehat{\lambda}_1+ \cdots +\widehat{\lambda}_t$, with associated 
rectangles $R_1,\ldots,R_t$ (see \eqref{eq:rectangledef} and Figure~\ref{fig:rectangular-decomposition}). 
Define a map of abelian groups
\[
L_{\lambda} \colon \Lambda^{(\le m)}_{\lambda} \to \Lambda^{(\le m)}_{\lambda}, \quad
L_{\lambda}(s_\mu) 
= \sum_{i = 1}^t i 
\sum_{ \substack{ \nu \supseteq \mu \\
		\nu \smallsetminus \mu \in R_i	\\
		|\nu| = |\mu| + 1 
}} s_\nu.
\]
For example, if $\lambda = \emptyset$ then $L_\lambda$ is  the zero map.
Consider the following ideal in $\Lambda^{(\le m)}$:
\[
E_\lambda^{(\le m)} := \bigoplus_{ \substack{ \mu \supseteq \lambda \\ \ell(\mu) = m \\ \mu \nsubseteq \lambda \cup R_1 \cup \cdots \cup R_t } } \Z \cdot s_\mu \subseteq  \Lambda^{(\le m)}_{\lambda}.
\]
Since $L_{\lambda}(E_\lambda^{(\le m)}) \subseteq E_\lambda^{(\le m)}$,  $L_{\lambda}$ induces a map of abelian groups
$$\widetilde{L}_{\lambda} \colon \Lambda^{(\le m)}_{\lambda}/E_\lambda^{(\le m)} \to \Lambda^{(\le m)}_{\lambda}/E_\lambda^{(\le m)}.$$ 
The following lemma will be crucial in our proof of Theorem~\ref{thm:main}.

\begin{lemma}\label{lem:commute}
	Let $\lambda$ be a  partition with $\ell(\lambda) = m$. 
	Then $\widetilde{L}_{\lambda}$ is a $\Lambda^{(\le m)}$-module homomorphism. 
\end{lemma}
\begin{proof}
	For $1 \le i \le t$, define a map of abelian groups
	\[
	L_{\lambda,i} \colon \Lambda^{(\le m)}_{\lambda} \to \Lambda^{(\le m)}_{\lambda}, \quad
	L_{\lambda,i}(s_\mu) 
	=
	\sum_{ \substack{ \nu \supseteq \mu \\
			\nu \smallsetminus \mu \in R_i \cup \cdots \cup R_t	\\
			|\nu| = |\mu| + 1 
	}} s_\nu.
	\]
	Then $L_\lambda = \sum_{i = 1}^t L_{\lambda,i}$. It is enough to show that
	$L_{\lambda,i}(g \star_m h) -  L_{\lambda,i}(g) \star_m h$ lies in $E_\lambda^{(\le m)}$ for any $1 \le i \le t$, $g \in \Lambda^{(\le m)}_{\lambda}$, and $h \in \Lambda^{(\le m)}$. 
	By linearity, we may assume that $g = s_\mu$ for some $\lambda \subseteq \mu$ with $\ell(\mu) = m$. Recall the surjective ring map $\pi_m \colon \Lambda \to \Lambda^{(\le m)}$. 
	Since $\Lambda$ is generated as a ring by $\{ h_j : j \in \Z_{\ge 0} \}$, $\Lambda^{(\le m)}$ is also generated by $\{ h_j : j \in \Z_{\ge 0} \}$. Since the lemma holds for $h = s_0 = s_\emptyset = 1$, we reduce to the case when 
	$h = h_j$ for some $j \in \Z_{> 0}$. 
	That is, we need to show that
	$$L_{\lambda,i}(s_\mu \star_m h_j) -  L_{\lambda,i}(s_\mu) \star_m h_j \in E_\lambda^{(\le m)}.$$
	On the one hand, 
	\[
	L_{\lambda,i}(s_\mu \star_m h_j)
	= \sum_{\ell(\nu) = m } c_{\mu,(j)}^{\nu} L_{\lambda,i}(s_\nu) 
	= \sum_{ \substack{\nu' \supseteq \mu \\ \ell(\nu') = m }}   \, \, 
	\sum_{ \substack{ \mu \subseteq \nu \subseteq \nu' \\
			\nu' \smallsetminus \nu \in R_i \cup \cdots \cup R_t
			\\
			|\nu'| = |\nu| + 1 
	}}
	c_{\mu,(j)}^{\nu}  s_{\nu'}.
	\]
	On the other hand,
	\[
	L_{\lambda,i}(s_\mu) \star_{m} h_j = \sum_{ \substack{ \nu \supseteq \mu \\
			\nu \smallsetminus \mu \in R_i \cup \cdots \cup R_t
			\\
			|\nu| = |\mu| + 1 
	}}
	s_\nu \star_{m}  s_{j} 
	=  \sum_{ \substack{\nu' \supseteq \mu \\ \ell(\nu') = m }}  \, \, 
	\sum_{ \substack{ \mu \subseteq \nu \subseteq \nu' \\
			\nu \smallsetminus \mu \in R_i \cup \cdots \cup R_t	\\
			|\nu| = |\mu| + 1 
	}}
	c_{\nu,(j)}^{\nu'}  s_{\nu'}.
	\]
	Consider a partition $\nu'$ with $\nu' \supseteq \mu$ and $\ell(\nu') = m$. 
	By definition, $s_{\nu'} \in E_\lambda^{(\le m)}$ unless  
	$\nu' \subseteq \lambda \cup R_1 \cup \cdots \cup R_t$. 
	It is therefore enough to show that if $\nu' \subseteq \lambda \cup R_1 \cup \cdots \cup R_t$, then 
	\begin{equation*}
		\sum_{ \substack{ \mu \subseteq \nu \subseteq \nu' \\
				\nu' \smallsetminus \nu \in R_i \cup \cdots \cup R_t \\
				|\nu'| = |\nu| + 1 
		}}
		c_{\mu,(j)}^{\nu} = \sum_{ \substack{ \mu \subseteq \nu \subseteq \nu' \\
				\nu \smallsetminus \mu \in R_i \cup \cdots \cup R_t
				\\
				|\nu| = |\mu| + 1 
		}}
		c_{\nu,(j)}^{\nu'}
	\end{equation*}
	Equivalently, by Pieri's rule,
	\begin{align*}
		&|\{ \mu \subseteq \nu \subseteq \nu' : \nu' \smallsetminus \nu \in R_i \cup \cdots \cup R_t, \, |\nu'| = |\nu| + 1, \, \nu \smallsetminus \mu \textrm{ is a horizontal } j\textrm{-strip}  \}| = 
		\\
		&|\{ \mu \subseteq \nu \subseteq \nu' :  \nu \smallsetminus \mu \in R_i \cup \cdots \cup R_t, \, |\nu| = |\mu| + 1, \, \nu' \smallsetminus \nu \textrm{ is a horizontal } j\textrm{-strip}  \}|.
	\end{align*}
	Note that both sides are zero unless the skew Young diagram $\nu' \smallsetminus \mu$ contains at most one vertical column of size at most $2$. 
		Recall that the connected components of a skew Young diagram are the maximal subsets of boxes that are connected via edge adjacency. Then $R_1,\ldots,R_t$ are the connected components of $R_1 \cup \cdots \cup R_t$.
	If $\nu' \smallsetminus \mu$ is a horizontal $(j + 1)$-strip, then both sides equal the number of connected components of $\nu' \smallsetminus \mu$ in $R_i \cup \cdots \cup R_t$. Otherwise, we may assume that $\nu' \smallsetminus \mu$ contains a unique vertical column $C$ of length $2$. Let $\alpha$ be the upper box of $C$, and let $\beta$ be the lower box of $C$. Then the left-hand side of the above equation is equal to $1$ if $\beta \in R_i \cup \cdots \cup R_t$ and is $0$ otherwise, while the right-hand side is equal to $1$ if $\alpha \in R_i \cup \cdots \cup R_t$ and is $0$ otherwise.  
	Since $C \subseteq \nu' \smallsetminus \mu \subseteq R_1 \cup \cdots \cup R_t$ is connected, we must have $C \subseteq R_{j'}$ for some $j'$, and the result follows.
\end{proof}

\section{Projective bundle rings}\label{sec:projbund}

In this section, we develop some properties of (multi)projective bundle rings. Let $A = A^0 \oplus \dotsb \oplus A^n$ be a commutative graded $\mathbb{R}$-algebra with an $\R$-linear map $\deg_A \colon A^n \to \mathbb{R}$ that satisfies Poincar\'{e} duality. It is convenient to allow $A$ to be the zero ring (even with $n > 0$); the zero ring vacuously satisfies Poincar\'{e} duality. The first few results that we prove will not require $A$ to satisfy the K\"{a}hler package. 

Choose some $r$ and classes $c_1, \dotsc, c_r$, with $c_i \in A^i$. Set $c_0 = 1$, and set $c_i = 0$ for $i > r$ or $i < 0$. Set $B = A[\zeta]/(\zeta^r - c_1 \zeta^{r-1} + \dotsb + (-1)^r c_r)$, and let $\pi^* \colon A \to B$ be the obvious inclusion. We quickly recall some standard properties of $B$; see \cite[Section 3.1]{LarsonPartida} or \cite[Section 3]{IMMSW}. It follows from the relation defining $B$ that there is a decomposition 
$$B \cong A \oplus \zeta A \oplus \dotsb \oplus \zeta^{r-1} A$$
as $A$-modules, giving \eqref{eq:Bdirectsum}. In particular, $B^{n + r - 1}$ is naturally isomorphic to $A^n$. Define a map $\deg_B \colon B^{n + r -1} \to \mathbb{R}$ via the formula $\deg_B(\zeta^{r-1} x) = \deg_A(x)$ for all $x \in A^n$. Extend $\deg_A$ and $\deg_B$ to be $0$ on the other graded pieces. 
There is an $A$-module homomorphism $\pi_* \colon B \to A$ which, in terms of the decomposition \eqref{eq:Bdirectsum}, maps $\zeta^{r-1} A$ to $A$ via the identity and sends the other summands to $0$. We call this the \emph{pushforward} map. 
The pushforward map is the adjoint of $\pi^*$: for any $x \in A$ and $y \in B$, we have
$$\deg_B(\pi^* (x) y) = \deg_A(x \pi_*(y)).$$
By \eqref{eq:Bdirectsum}, an element $y \in B$ is zero if and only if $\pi_*(y\zeta^i) = 0$ for all nonnegative integers $i$. 

We will need analogues of Segre classes. Define classes $s_i \in A^i$ for $i \in \{0, \dotsc, n\}$ via the relation
\begin{equation}\label{eq:segre}
(1 - c_1 + c_2 - \dotsb + (-1)^r c_r)(s_0 + s_1 + \dotsb + s_n) = 1 \text{ in }A.
\end{equation}
For example, $s_0 = 1$, $s_1 = c_1$, and $s_2 = c_1^2 - c_2$.  Set $s_i = 0$ for $i > n$ or $i < 0$.
With the notation of Section~\ref{sec:symmetric}, we have a graded ring homomorphism 
 $$\Psi \colon \Lambda = \Z[e_1,e_2,\ldots] \to A, \quad \Psi(e_i) = c_i.$$
 For example, $\Psi(h_i) = \Psi(s_{(i)}) = s_i$ is the $i$th Segre class. 
In general, for a partition $\lambda$,  analogously to \eqref{eq:schurdef}, 
 $$\Psi(s_\lambda) = \det(s_{\lambda_i + j -i})_{i,j=1}^{\ell(\lambda)}.$$  
 If $\lambda'$ is the transpose of $\lambda$, then we have the alternative expression $\Psi(s_\lambda) = \det(c_{\lambda_i' + j -i})_{i,j=1}^{\ell(\lambda')}$. It follows that $\Psi(s_{\lambda}) = 0$ if $\ell(\lambda) > r$.  
 For a nonnegative integer $m$, $\Psi$ restricts to a map $\Lambda^{(\le m)} \to A$ of abelian groups. Importantly, if $m < r$, then this is \emph{not} a ring homomorphism when we consider $\Lambda^{(\le m)}$ as a ring with the product $\star_m$ that we discussed in Section~\ref{sec:symmetric}.
 
In what follows,  we will often implicitly view elements of $\Lambda$ as elements of $A$ via $\Psi$. In particular, we will often
write $s_\lambda$ for $\Psi(s_\lambda) \in A^{|\lambda|}$, and, 
if $\ell(\lambda), \ell(\mu) \le m$, then we often write $s_\lambda \star_m s_\mu$ for $\Psi(s_\lambda \star_m s_\mu) \in A$.

\begin{lemma}\label{lem:powers}
For $k \ge 1$, we have that
$$\zeta^{r - 1 + k} = \sum_{i=0}^{r-1} (-1)^{i} s_{(k, 1^{i})} \zeta^{r - 1 - i}.$$
\end{lemma}

\begin{proof}
We prove this by induction on $k$. As $s_{(1^i)} = c_i$, the case when $k=1$ says that
$$\zeta^r = c_1 \zeta^{r-1} - c_2 \zeta^{r-2} + \dotsb + (-1)^{r - 1} c_r,$$
which is exactly the relation defining $B$. For the induction step, we have
$$\zeta \cdot \zeta^{r - 1 + k} = s_{(k)} \zeta^r + \sum_{i=1}^{r-1} (-1)^{i} s_{(k, 1^{i})} \zeta^{r - i} = s_{k} \cdot (c_1 \zeta^{r-1} - c_2 \zeta^{r-2} + \dotsb + (-1)^{r - 1} c_r) + \sum_{i=1}^{r-1} (-1)^{i} s_{(k, 1^{i})} \zeta^{r - i}.$$
By Pieri's rule, $s_k \cdot c_i = s_{(k, 1^i)} + s_{(k + 1, 1^{i - 1})}$. Expanding implies the result. 
\end{proof}

In particular, we have the following formula for the pushforward map: for all $i \in \Z$,

\begin{equation}\label{eq:pushforward}
\pi_*(\zeta^{r - 1 + i}) = s_i.
\end{equation}

Finally, we observe that projective bundle rings always satisfy Poincar\'{e} duality if the base ring does; see \cite[Proposition 3.1]{LarsonPartida} or \cite[Lemma 3.2]{IMMSW}. 

\begin{proposition}\label{prop:PDprojbun}
The ring $B$ satisfies Poincar\'{e} duality. 
\end{proposition}

\begin{proof}
Choosing a basis for each $A^i$ and using the decomposition \eqref{eq:Bdirectsum}, the matrix representing the Poincar\'{e} pairing is block triangular, and the blocks are nondegenerate by Poincar\'{e} duality for $A$. 
\end{proof}

\subsection{Multiprojective bundle rings}
We now discuss a generalization of projective bundle rings. For the rest of this section, we fix some  $m \ge 1$ and set 
$$C = \frac{A[\zeta_1, \dotsc, \zeta_m]}{(\zeta_1^{r} - c_1 \zeta_1^{r-1} + \dotsb + (-1)^r c_r, \dotsc, \zeta_m^r - c_1 \zeta_{m}^{r-1} + \dotsb + (-1)^r c_r)}.$$
In particular, $C$ is an iterated projective bundle ring over $A$. Let $p^* \colon A \to C$ be the natural map. We have a decomposition
\begin{equation*}
C = \bigoplus_{0 \le i_1, \dotsc, i_m \le r-1} \zeta_1^{i_1} \dotsb \zeta_m^{i_m} A.
\end{equation*}

From the structure of $C$ as an iterated projective bundle ring, we obtain a map $\deg_C \colon C^{n + m(r-1)} \to \mathbb{R}$. 
Repeatedly applying Proposition~\ref{prop:PDprojbun}, we see that $C$ satisfies Poincar\'{e} duality. Let $p_* \colon C \to A$ be the pushforward map, obtained by pushing forward along the iterated projective bundle ring structure (in any order). As before, $p_*$ is the adjoint to $p^*$ under the pairings on $A$ and $C$ induced by Poincar\'{e} duality. Applying \eqref{eq:pushforward}, we have
\begin{equation}\label{eq:pushforwardmulti}
p_*(\zeta_1^{r - 1 + i_1} \dotsb \zeta_m^{r - 1 + i_m}) = s_{i_1} \dotsb s_{i_m}.	
\end{equation}
An element $y \in C$ is zero if and only if $p_*(y\zeta_1^{i_1}\cdots \zeta_m^{i_m}) = 0$ for all $(i_1,\ldots,i_m) \in \N^m$. 

There is an action of the symmetric group $S_m$ on $C$ by permuting $\zeta_1, \dotsc, \zeta_m$. The map $p_*$ is $S_m$-equivariant with respect to the trivial action of $S_m$ on $A$. 
Recall that
$$\Delta = \prod_{1 \le i < j \le m} (\zeta_i - \zeta_j) \in C^{\binom{m}{2}}.$$
Let $C^{S_m}$ denote the subring of $S_m$-invariant elements of $C$. 
For a permutation $\sigma \in S_m$, we use $\sgn(\sigma) \in \{ \pm 1 \}$ to denote its sign. 
Say that an element $y \in C$ is \emph{antisymmetric} if $\sigma \cdot y = \sgn(\sigma) y$ for all $\sigma \in S_m$. Then the antisymmetric elements of $C$ are precisely $C^{S_m} \cdot \Delta$. 

\begin{lemma}\label{lem:antisymmetriczero}
	Consider an element $y \in C^{S_m} \cdot \Delta$. Then $y = 0$ in $C$ if and only if $p_*(yz \Delta) = 0$ in $A$ for all $z \in C^{S_m}$. 
\end{lemma}
That is, an antisymmetric class in $C$ is nonzero if and only if there is another antisymmetric class such that the pushforward of their product is nonzero. 
\begin{proof}[Proof of Lemma~\ref{lem:antisymmetriczero}]
	Assume that $p_*(yz  \Delta) = 0$ in $A$ for all $z \in C^{S_m}$. We have seen that $y = 0$ if and only if $p_*(xy) = 0$ for all $x \in C$. Since $p_*$ is $S_m$-equivariant and $y$ is antisymmetric, we have
	\begin{align*}
		p_*(xy) = \frac{1}{m!}\sum_{\sigma \in S_m} p_*(\sigma \cdot (xy)) 
		= \frac{1}{m!}\sum_{\sigma \in S_m} \sgn(\sigma)p_*((\sigma \cdot x)y) 
		= p_*(x'y), 
	\end{align*}
	where $x' = \frac{1}{m!}\sum_{\sigma \in S_m} \sgn(\sigma) \sigma \cdot x$. Since $x'$ is antisymmetric, we have $p_*(x'y) = 0$ by assumption. We conclude that $y = 0$. The converse is clear. 
\end{proof}

For a partition $\lambda$, 
recall that the Schur polynomial $s_\lambda(x_1,\ldots,x_m) \in \Z[x_1,\ldots,x_m]^{S_m}$ is a symmetric polynomial that is homogeneous of degree $|\lambda|$. 
Let $s_{\lambda}(\zeta_1, \dotsc, \zeta_m) \in C^{S_m}$ be the Schur polynomial of $\lambda$ evaluated at $\zeta_1, \dotsc, \zeta_m$. 
Equivalently, $s_{\lambda}(\zeta_1, \dotsc, \zeta_m)$ is the Schur class associated to $\lambda$ for the projective bundle ring 
$C[\zeta]/((\zeta - \zeta_1)\cdots(\zeta - \zeta_m))$. 
We now prove the key formula for Schur classes in terms of multiprojective bundle rings that was stated in the introduction. 
Recall that Proposition~\ref{prop:pushforward} states that if 
$\lambda$ is a partition, and $\ell(\lambda) \le m \le r$, then
$$(-1)^{\binom{m}{2}} m! s_{\lambda} = p_* \left(\Delta^2 (\zeta_1 \dotsb \zeta_m)^{r-m} s_{\lambda}(\zeta_1, \dotsc, \zeta_m)\right ).$$

\begin{proof}[Proof of Proposition~\ref{prop:pushforward}]
We claim that 
\begin{equation}\label{eq:monomialpush}
p_*(\Delta s_{\lambda}(\zeta_1, \dotsc, \zeta_m) (\zeta_1 \dotsb \zeta_m)^{r-m} \zeta_2 \zeta_3^2 \dotsb \zeta_m^{m-1}) = s_{\lambda}.
\end{equation}
Assume this claim. 
For any $\sigma \in S_m$, since $p_*$ is $S_m$-equivariant and $\Delta$ is antisymmetric,  we have
\begin{align*}
s_\lambda &=	p_*(\sigma \cdot (\Delta s_{\lambda}(\zeta_1, \dotsc, \zeta_m) (\zeta_1 \dotsb \zeta_m)^{r-m} \zeta_2 \zeta_3^2 \dotsb \zeta_m^{m-1})) 
\\
	&= \sgn(\sigma) p_*(\Delta s_{\lambda}(\zeta_1, \dotsc, \zeta_m) (\zeta_1 \dotsb \zeta_m)^{r-m} \zeta_{\sigma(1)}^0 \zeta_{\sigma(2)}^1 \zeta_{\sigma(3)}^2 \dotsb \zeta_{\sigma(m)}^{m-1}) 
\end{align*}
Summing the above equation over all $\sigma$ in $S_m$ gives
\[
m! s_\lambda = p_*(\Delta s_{\lambda}(\zeta_1, \dotsc, \zeta_m) (\zeta_1 \dotsb \zeta_m)^{r-m}  \det(\zeta_i^{j - 1})_{i, j=1}^m). 
\]	
Observe that $\Delta = \det(\zeta_i^{m - j})_{i, j=1}^m = (-1)^{\binom{m}{2}}\det(\zeta_i^{j - 1})_{i, j=1}^m$. Substituting this expression into the above equation gives the result. It remains to prove the claim \eqref{eq:monomialpush}. 

By \cite[Theorem 7.15.1]{stanley-ec2}, we see that
$$\Delta s_{\lambda}(\zeta_1, \dotsc, \zeta_m) = \det(\zeta_i^{\lambda_j + m - j})_{i, j=1}^m.$$
Expanding the determinant and pushing forward using \eqref{eq:pushforwardmulti}, we see that
\begin{align*}
	p_*(\Delta s_{\lambda}(\zeta_1, \dotsc, \zeta_m) (\zeta_1 \dotsb \zeta_m)^{r-m} \zeta_2 \zeta_3^2 \dotsb \zeta_m^{m-1}) &= p_*\left(\sum_{\sigma \in S_m} \sgn(\sigma) \prod_{j=1}^{m} \zeta_{\sigma(j)}^{\lambda_{j} + m - j + (r - m) + (\sigma(j) - 1)}\right)  \\
	&= \sum_{\sigma \in S_m} \sgn(\sigma) p_*\left(\prod_{j=1}^{m} \zeta_{\sigma(j)}^{r - 1 + \lambda_{j}  - j  + \sigma(j)}\right)  \\
	&= \sum_{\sigma \in S_m} \sgn(\sigma) \prod_{j=1}^{m} s_{\lambda_{j}  - j  + \sigma(j)}  \\
	&= \det(s_{\lambda_{j} - j + i})_{i, j=1}^m
	\\
	&= \det(s_{\lambda_{i} - i + j})_{i, j=1}^m
	\\
	&= s_{\lambda}. \qedhere
\end{align*}
\end{proof}
\begin{remark}
There are a number of formulas describing Schur classes as pushforwards; see, e.g., \cite[Chapter 14]{FultonIntersection} and \cite{DarondeauPragacz}. As far as the authors are aware, neither Proposition~\ref{prop:pushforward} nor \eqref{eq:monomialpush} has appeared in the literature before. It would be interesting to connect these formulas with existing formulas. 
\end{remark}

We will also need the following lemma. 

\begin{lemma}\label{lem:detectzeroupstairs}
	Fix $a \in A$ and a partition $\mu$ with $\ell(\mu) \le m \le r$. Then
	$a \Delta (\zeta_1 \cdots \zeta_m)^{r - m} s_\mu(\zeta_1,\ldots,\zeta_m) = 0$ in $C$ if and only if 
	$a (s_\mu \star_m s_\nu) = 0$ in $A$ for all partitions $\nu$  with $\ell(\nu) \le m$. 
\end{lemma}
\begin{proof}
	Let $y = a \Delta (\zeta_1 \cdots \zeta_m)^{r - m} s_\mu(\zeta_1,\ldots,\zeta_m) \in C$. Then $y$ is antisymmetric, so by Lemma~\ref{lem:antisymmetriczero}, $y = 0$ if and only if
	$p_*(y  z \Delta) = 0$ in  $A$ for all $z \in C^{S_m}$. This holds if and only if $p_*(y  s_\nu(\zeta_1,\ldots,\zeta_m) \Delta) = 0$ in $A$ for all partitions $\nu$ with $\ell(\nu) \le m$. Substituting in the definition of $y$ and then using Proposition~\ref{prop:pushforward}, we compute
	\begin{align*}
		p_*(y  s_\nu(\zeta_1,\ldots,\zeta_m) \Delta) &= a p_*(\Delta^2 (\zeta_1 \cdots \zeta_m)^{r - m} s_\mu(\zeta_1,\ldots,\zeta_m) s_\nu(\zeta_1,\ldots,\zeta_m)) \\
		&= a \sum_{\ell(\nu') \le m } c_{\mu,\nu}^{\nu'} 
		p_*(\Delta^2 (\zeta_1 \cdots \zeta_m)^{r - m}  s_{\nu'}(\zeta_1,\ldots,\zeta_m)) 
		\\
		&= ((-1)^{\binom{m}{2}} m!) a \sum_{\ell(\nu') \le m } c_{\mu,\nu}^{\nu'} s_{\nu'}
		\\
		&= ((-1)^{\binom{m}{2}} m!) a (s_\mu \star_m s_\nu).
		\\
	\end{align*}
The result follows.
\end{proof}

\subsection{Positivity of multiprojective bundle rings}
For the remainder of this section, we assume that $A$ has the K\"{a}hler package with respect to a nonempty open convex cone $\mathcal{K}_A \subseteq A^1$. Recall that $C$ is the $m$-fold multiprojective bundle ring and that 
 $\mathcal{K}_C = p^* (\mathcal{K}_A) + \mathbb{R}_{>0} \zeta_1 + \dotsb + \mathbb{R}_{>0} \zeta_m$.

We first recall a general result about the K\"{a}hler package.
Let $D = D^0 \oplus \dotsb \oplus D^s$ be a commutative graded $\mathbb{R}$-algebra with an $\R$-linear map $\deg_D: D^s \to \R$. Let $U$ be the subset of $D^1$ consisting of elements $h$ such that, for all $i < s/2$, multiplication by $h^{s - 2i}$ is an isomorphism from $D^{i}$ to $D^{s - i}$, i.e., the locus where the hard Lefschetz property holds. Then $U$ is a (possibly empty) open cone. An element $h \in U$ satisfies the Hodge--Riemann relations if and only if, for all $i \le s/2$, the symmetric bilinear form on $D^i$ given by $(x, y) \mapsto \deg_D(h^{s - 2i}xy)$ has signature $\sum_{j \le i} (-1)^j (\dim D^j - \dim D^{j-1})$. In particular, given a connected component $U_0$ of $U$, either the Hodge--Riemann relations hold for every $h \in U_0$, or the Hodge--Riemann relations fail for every $h \in U_0$. 
Let $U_{D,\HR}$ be the locus of elements of $D^1$ satisfying the Hodge--Riemann relations. 
We conclude that $U_{D,\HR}$
is an open cone which is a union of some of the connected components of $U$. 
The following result will be crucial in the proof of Theorem~\ref{thm:multiproj}. 

\begin{proposition}\cite[Proposition 1.2.2]{KK87}\label{prop:convex}
	Let $D = D^0 \oplus \dotsb \oplus D^s$ be a commutative graded $\mathbb{R}$-algebra  with an $\R$-linear map $\deg_D: D^s \to \R$. 
	Then every connected component of $U_{D,\HR}$ is convex.
\end{proposition}

We now show that it is automatic that $C$ has the K\"{a}hler package with respect to some inexplicit cone.

\begin{proposition}\label{prop:somecone}
There is a nonempty open convex cone $\mathcal{K} \subseteq \mathcal{K}_C$ such that $C$ has the K\"{a}hler package with respect to $\mathcal{K}$. 
\end{proposition}

In order to prove this, we first 
recall the case when all $c_i$ vanish. 

\begin{proposition}\label{prop:civanishexample}\cite{AHK18}*{Proposition~7.7}
	If $c_i = 0$ for $1 \le i \le r$, then $C = A[\zeta_1,\ldots,\zeta_m]/(\zeta_1^r,\ldots,\zeta_m^r)$ has the K\"{a}hler package with respect to $\mathcal{K}_C$. 
\end{proposition}

Next we introduce some temporary notation. For $0 \le i \le m$, let $C_i$ denote the $i$-fold multiprojective bundle with degree map $\deg_{C_i}$ and cone $\K_{C_i}$. For example, $C_0 = A$ and $C_m = C$.

\begin{lemma}\label{lem:deformationKahler}
	Let $U_{C,\HR}$ be the open cone consisting of elements of $C^1$  that satisfy the Hodge--Riemann relations. 
	Consider a subset $I \subseteq \{ 1,\ldots, m \}$ and suppose that 
	$C_{|I|}$ has the  K\"{a}hler package with respect to $\mathcal{K}_{C_{|I|}}$. 
	Consider an element $h \in \K_A$ and the corresponding affine space $h + \sum_{i = 1}^m \R \zeta_i \subseteq C^1$.  
	Then there is an open set $U \subseteq h + \sum_{i = 1}^m \R \zeta_i$ satisfying the following properties:
	\begin{enumerate}
		\item $U \cap (h + \sum_{i = 1}^m \R_{> 0} \zeta_i) \subseteq U_{C,\HR}$,
		\item $U \cap (h + \sum_{i = 1}^m \R_{> 0} \zeta_i)$ is path connected,
		\item $h + \sum_{i \in I} \R_{> 0} \zeta_i \subseteq U$. 
	\end{enumerate}
	\end{lemma}
\begin{proof}
	Let $\mu = (\mu_1,\ldots,\mu_m) \in \R_{\ge 0}^m$. The support of $\mu$, denoted $\Supp(\mu)$, is $\{ i : \mu_i \neq 0 \} \subseteq \{ 1, \ldots, m \}$.
	The following construction is inspired by  \cite{IMMSW}*{Theorem~8.3}. 
	Consider the iterated projective bundle
	\[
	D_\mu := \frac{A[\zeta_1, \dotsc, \zeta_m]}{(\sum_{j = 0}^r (-1)^j c_j \mu_i^j \zeta_i^{r - j} : 1 \le i \le m)},
	\]
with corresponding degree map $\deg_{D_\mu} \colon D_\mu^{n + m(r - 1)} \to \R$.
Consider the iterated projective bundle
\[
D_\mu' := \frac{A[\zeta_1, \dotsc, \zeta_m]}{(\sum_{j = 0}^r (-1)^j c_j \zeta_i^{r - j} : i \in \Supp(\mu)) + (\zeta_i^r : i \notin \Supp(\mu))},
\]
with corresponding degree map $\deg_{D_\mu'} \colon (D_\mu')^{n + m(r - 1)} \to \R$. Then we have an isomorphism of graded $\R$-algebras
\[
\phi_\mu \colon D_\mu \to D_\mu', \quad 
\phi_\mu(\zeta_i) = \begin{cases}
	\mu_i \zeta_i &\textrm{if } i \in \Supp(\mu), \\
	\zeta_i &\textrm{otherwise. }  \\
\end{cases}
\]
Moreover, $\deg_{D_\mu}$ is equal to the composition of  $\deg_{D_\mu'}$ and  $\phi_\mu$ up to multiplication by a positive scalar, and $\phi_\mu(\K_{D_\mu}) = \K_{D_\mu'}$. 

Consider $\rho = (\rho_1,\ldots,\rho_m) \in \mathbb{R}^m_{\ge 0}$ with support $I$. By definition, $D_\rho' \cong C_{|I|}[z_1,\ldots,z_{m - |I|}]/(z_1^r,\ldots,z_{m - |I|}^r)$. By assumption, $C_{|I|}$ has the  K\"{a}hler package with respect to $\mathcal{K}_{C_{|I|}}$. Proposition~\ref{prop:civanishexample} then implies that $D_{\rho}'$ has the K\"{a}hler package with respect to $\mathcal{K}_{D_{\rho}'}$, and we deduce that $D_{\rho}$ has the K\"{a}hler package with respect to $\mathcal{K}_{D_{\rho}}$. 
In particular, $D_\rho$ satisfies the Hodge--Riemann relations with respect to $h + \zeta_1 + \cdots + \zeta_m \in \mathcal{K}_{D_{\rho}}$. 

Fix a homogeneous $\R$-basis $\{ a_{i,j} : 1 \le j \le \dim A^i \}$ for each $A^i$. 
By repeatedly applying \eqref{eq:Bdirectsum}, $\{ a_{i,j} \zeta_1^{k_1}\cdots \zeta_m^{k_m} : 0 \le i \le n, \, 1 \le j \le \dim A^i, \, 0 \le k_1,\ldots, \, k_m \le r - 1 \}$ is an $\R$-basis for $D_\mu$. With respect to this basis, in any given degree, the Hodge--Riemann form associated to $h + \zeta_1 + \cdots + \zeta_m$ is a matrix whose entries are polynomials in $\mu_1,\ldots,\mu_m$. It follows that there is an open neighborhood $W_\rho$ of $\rho$ in $\R^m$ such that 
$D_\mu$ satisfies the Hodge--Riemann relations with respect to $h + \zeta_1 + \cdots + \zeta_m$ for all 
$\mu \in W_\rho \cap \R_{\ge 0}^m$. 

Consider some 
$\mu \in W_\rho \cap \R_{> 0}^m$. 
Since $\Supp(\mu) = \{ 1,\ldots, m\}$, we have $D_\mu' = C$. Applying the isomorphism $\phi_\mu \colon D_\mu \to C$, we see that $h + \mu_1\zeta_1 + \cdots + \mu_m \zeta_m \in U_{C,\HR}$. Let 
$$U_\rho = \{ h + \mu_1\zeta_1 + \cdots + \mu_m \zeta_m : (\mu_1,\ldots,\mu_m) \in W_\rho \}, \quad \text{and} \quad U = \bigcup_{\substack{\rho \in \R^m_{\ge 0} \\ \Supp(\rho) = I}} U_\rho.$$ 
Then  $U$ is an open subset of $h + \sum_{i = 1}^m \R \zeta_i$ such that
$U \cap (h + \sum_{i = 1}^m \R_{> 0} \zeta_i) \subseteq U_{C,\HR}$. Also, by construction, $h + \sum_{i \in I} \R_{> 0} \zeta_i \subseteq U$. It remains to show that
$U \cap (h + \sum_{i = 1}^m \R_{> 0} \zeta_i)$ is path connected. 
Consider $\rho, \rho' \in \R^m_{\ge 0}$ with $\Supp(\rho) = \Supp(\rho') = I$. Consider the compact interval $[\rho,\rho'] \subseteq \{ \mu \in \R^m_{\ge 0} : \Supp(\mu) = I \}$. Then there is a finite sequence $\rho = \mu^1, \mu^2, \ldots, \mu^s = \rho'$ of elements in $[\rho,\rho']$ such that 
$U_{\mu^{i - 1}} \cap U_{\mu^{i}} \cap (h + \sum_{i = 1}^m \R_{> 0} \zeta_i) \neq \emptyset$ for $1 < i \le s$.  Since each $U_{\mu^i} \cap (h + \sum_{i = 1}^m \R_{> 0} \zeta_i)$ is path connected, it follows that there is a path in $U$ from any point in $U_{\rho} \cap (h + \sum_{i = 1}^m \R_{> 0} \zeta_i)$ to any point in $U_{\rho'} \cap (h + \sum_{i = 1}^m \R_{> 0} \zeta_i)$. 
\end{proof}

We now prove Proposition~\ref{prop:somecone} and Theorem~\ref{thm:multiproj}. We continue with the notation from Lemma~\ref{lem:deformationKahler}. 

\begin{proof}[Proof of Proposition~\ref{prop:somecone}]
	Applying Lemma~\ref{lem:deformationKahler} with $I = \emptyset$ implies that $U_{C,\HR} \cap \K_C \neq \emptyset$. Let $\K$ be any nonempty open convex subcone of $U_{C,\HR} \cap \K_C$; then $C$ has the K\"{a}hler package with respect to $\mathcal{K}$.
\end{proof}

\begin{proof}[Proof of Theorem~\ref{thm:multiproj}]
	
Suppose that $B$ satisfies the K\"{a}hler package with respect to $\mathcal{K}_B$. We need to show that $C$ has the K\"{a}hler package with respect to $\mathcal{K}_C$.  
Fix an element $h \in \K_A$.

Consider a subset $I \subseteq \{ 1, \ldots, m \}$ with $|I| \le 1$. Then Lemma~\ref{lem:deformationKahler} implies that 
 there is an open set $U_I \subseteq h + \sum_{i = 1}^m \R \zeta_i$ satisfying the following properties:
\begin{enumerate}
	\item $U_I \cap (h + \sum_{i = 1}^m \R_{> 0} \zeta_i) \subseteq U_{C,\HR}$,
	\item $U_I \cap (h + \sum_{i = 1}^m \R_{> 0} \zeta_i)$ is path connected,
	\item $h + \sum_{i \in I} \R_{> 0} \zeta_i \subseteq U_I$. 
\end{enumerate}
Let 
$$U = \bigcup_{\substack{I \subseteq \{ 1, \ldots, m \} \\ |I| \le 1}} U_I.$$
Then $U \cap (h + \sum_{i = 1}^m \R_{> 0} \zeta_i) \subseteq U_{C,\HR}$. Since $U_\emptyset \cap U_{\{ j \}} \cap (h + \sum_{i = 1}^m \R_{> 0} \zeta_i)$ is nonempty for $1 \le j \le m$, 
we see that $U \cap (h + \sum_{i = 1}^m \R_{> 0} \zeta_i)$ is path connected. Hence $U \cap (h + \sum_{i = 1}^m \R_{> 0} \zeta_i)$ is contained in a connected component of $U_{C,\HR}$. 
Then Proposition~\ref{prop:convex} implies that the convex hull 
$\Conv{U \cap (h + \sum_{i = 1}^m \R_{> 0} \zeta_i)}$ of $U \cap (h + \sum_{i = 1}^m \R_{> 0} \zeta_i)$ is contained in $U_{C,\HR}$. 

It remains to show that $\Conv{U \cap (h + \sum_{i = 1}^m \R_{> 0} \zeta_i)} = h + \sum_{i = 1}^m \R_{> 0} \zeta_i$. Since $h + \sum_{i = 1}^m \R_{> 0} \zeta_i$ is convex, we have $\Conv{U \cap (h + \sum_{i = 1}^m \R_{> 0} \zeta_i)} \subseteq h + \sum_{i = 1}^m \R_{> 0} \zeta_i$. Conversely,
consider an element $h + \mu_1 \zeta_1 + \cdots + \mu_m \zeta_m \in h + \sum_{i = 1}^m \R_{> 0} \zeta_i$. 
Consider $0 < \varepsilon \ll 1$. 
Let $$\rho_j = h + m (\mu_j - \varepsilon) \zeta_j + \varepsilon (\zeta_1 + \cdots + \zeta_m) \in U_{\{ j \}} \cap (h + \sum_{i = 1}^m \R_{> 0} \zeta_i)$$ for $1 \le j \le m$. 
Then $h + \mu_1 \zeta_1 + \cdots + \mu_m \zeta_m = \frac{1}{m}(\rho_1 + \cdots + \rho_m)$ lies in $\Conv{U \cap (h + \sum_{i = 1}^m \R_{> 0} \zeta_i)}$.
\end{proof}

\begin{remark}\label{rem:multiple}
The proof of Theorem~\ref{thm:multiproj} shows the following stronger statement. Suppose that, for $1 \le j \le m$, we have integers $r_j$ and classes $c_{j, 1}, c_{j,2}, \dotsc, c_{j, r_j}$ such that, for each $j$, the ring
$$B_j = \frac{A[\zeta]}{(\zeta^{r_j} - c_{j,1} \zeta^{r_j - 1} + \dotsb + (-1)^{r_j} c_{j, r_j})},$$
with corresponding inclusion $\pi_j^*: A \to B_j$
satisfies the K\"{a}hler package with respect to 
$K_{B_j} = \pi_j^* (\mathcal{K}_A) + \mathbb{R}_{>0} \zeta$.
Then the ring
$$\frac{A[\zeta_1, \dotsc, \zeta_m]}{(\zeta_1^{r_1} - c_{1,1} \zeta_1^{r_1 - 1} + \dotsb + (-1)^{r_1} c_{1, r_1}, \dotsc, \zeta_m^{r_m} - c_{m, 1} \zeta_m^{r_m - 1} + \dotsb + (-1)^{r_m} c_{m, r_m})}$$
satisfies the K\"{a}hler package with respect to the cone given by positive linear combinations of the images of elements of $\mathcal{K}_A$ and $\zeta_1, \dotsc, \zeta_m$. 
\end{remark}

We now recall another result about the K\"{a}hler package that will be crucial in what follows. It was first proved in \cite[Lemma 1.16]{CKS87}, where it is called the descent lemma, and in \cite[Theorem 2.1.5]{KK87}. See \cite[Theorem 3.2]{Cat08} and \cite[Lemma 4.5]{LefschetzModule} for the formulation which appears here. 

Let $D = D^0 \oplus \dotsb \oplus D^s$ be a commutative graded $\mathbb{R}$-algebra which is equipped with an $\R$-linear map $\deg_D \colon D^s \to \mathbb{R}$ that satisfies Poincar\'{e} duality. For $x \in D$, let $\operatorname{ann}(x)$ denote the annihilator of $x$. Assume that $x$ lies in $D^t$ for some $t$, and let $\varphi \colon D \to D/\operatorname{ann}(x)$ be the natural projection. There is an induced degree map $\deg_{D/\operatorname{ann}(x)}$ from the degree $s - t$ part of $D/\operatorname{ann}(x)$ to $\mathbb{R}$ which satisfies
$$\deg_{D/\operatorname{ann}(x)}(\varphi(y)) = \deg_D(xy)$$
for all $y \in D^{s - t}$. Then $D/\operatorname{ann}(x)$ satisfies Poincar\'{e} duality because Poincar\'{e} duality for $D$ implies that $\operatorname{ann}(x)$ is the kernel of the pairing $(a, b) \mapsto \deg_D(x ab)$. Note that this holds even if $x=0$. 

\begin{proposition}\label{prop:descent}
Suppose that $D$ has the K\"{a}hler package with respect to an open convex cone $\mathcal{K}_D$. Then for any $\ell \in \overline{\mathcal{K}}_D$, $D/\operatorname{ann}(\ell)$ satisfies the K\"{a}hler package with respect to the image of $\mathcal{K}_D$. 
\end{proposition}

We will often apply this result repeatedly to a product of several elements of $\overline{\mathcal{K}}_D$. We will frequently use that, for elements $\ell_1, \dotsc, \ell_t \in \overline{\mathcal{K}}_D$, $D/\operatorname{ann}(\ell_1 \dotsb \ell_t)$ satisfies the K\"{a}hler package with respect to the image of $\mathcal{K}_D$, and for any $y \in D^{s - t}$, we have
$$\deg_D(y \ell_1 \dotsb \ell_t) = \deg_{D/\operatorname{ann}(\ell_1 \dotsb \ell_t)}(y).$$
We now prove a variation on the fact that a direct sum of nef line bundles is a nef vector bundle. 

\begin{proposition}\label{prop:nefdirectsum}
Let $\ell_1, \dotsc, \ell_r$ be elements of $\overline{\mathcal{K}}_A$. Then $B = A[\zeta]/((\zeta - \ell_1)\dotsb (\zeta - \ell_r))$ satisfies the K\"{a}hler package with respect to $\K_B$. 
\end{proposition}

\begin{proof}
We prove the result by induction on $n$ and then by induction on $r$. The case when $n=0$ or when $r=0$ is trivial. By Proposition~\ref{prop:somecone}, we know that there is some open subcone $\mathcal{K}$ of $\mathcal{K}_B$ on which the K\"{a}hler package holds. It therefore suffices to show that the hard Lefschetz property holds for any element of the form $h + \zeta$, where $h \in \mathcal{K}_A$. 

Choose some $i < (n + r - 1)/2$ and some $b$ in the degree $i$ part of $B$ which is in the kernel of multiplication by $(\zeta + h)^{n + r - 1 - 2i}$. Note that $\zeta + h = (\zeta - \ell_r) + (h + \ell_r)$. We have
$$B/\operatorname{ann}(\zeta - \ell_r) \cong \frac{A[\zeta]}{((\zeta - \ell_1)\dotsb (\zeta - \ell_{r-1}))}.$$ 
By induction on $r$, the right-hand side satisfies the K\"{a}hler package with respect to 
the image of $\K_B$. 
Furthermore, the image of $b$ in this quotient is killed by $(\zeta + h)^{n + r - 1 - 2i}$. The Hodge--Riemann relations then imply that 
$$(-1)^i \deg_B(b^2 (\zeta + h)^{n + r - 2 - 2i}(\zeta - \ell_r)) \ge 0,$$
with equality if and only if $(\zeta - \ell_r) b = 0$. We also have
$$B/ \operatorname{ann}(h + \ell_r) \cong \frac{A/\operatorname{ann}(h + \ell_r)[\zeta]}{((\zeta - \ell_1)\dotsb (\zeta - \ell_{r}))}.$$
By induction on $n$ and Proposition~\ref{prop:descent}, the right-hand side satisfies the K\"{a}hler package with respect to the image of 
$\K_B$.
The Hodge--Riemann relations then imply that 
$$(-1)^i \deg_B(b^2 (\zeta + h)^{n + r - 2 - 2i}(h + \ell_r)) \ge 0,$$ 
with equality if and only if $(h + \ell_r) b = 0$. As 
$$\deg_B(b^2 (\zeta + h)^{n + r - 2 - 2i}(\zeta - \ell_r)) + \deg_B(b^2 (\zeta + h)^{n + r - 2 - 2i}(h + \ell_r)) = \deg_B(b^2 (\zeta + h)^{n + r - 1 - 2i}) = 0,$$
we see that $b (\zeta - \ell_r) = b(h + \ell_r) = 0$.  As $h + \ell_r$ satisfies the hard Lefschetz property on $A$, we see that multiplication by $h + \ell_r$ induces an injection from $A^s$ to $A^{s+1}$ for $s < n/2$. Writing $b$ in terms of the decomposition \eqref{eq:Bdirectsum}, the fact that $b(h + \ell_r) = 0$ implies that the $\zeta^j A^{i - j}$ component of $b$ is $0$ unless $j \le i - n/2$.  But if $b$ is nonzero, this contradicts that $b (\zeta - \ell_r) = 0$ as $i - n/2 < (r-1)/2$. This implies that $b = 0$, so $\zeta + h$ has the hard Lefschetz property, as desired. 
\end{proof}

\section{Positivity of Schur classes}\label{sec:positivity}

In this section, we prove Theorem~\ref{thm:main}. We continue with the notation from the previous sections.

\subsection{Proof of Theorem~\ref{thm:main}}

We assume the notation from Section~\ref{sec:symmetric}. 
Recall from Section~\ref{sec:projbund} that we have a ring homomorphism $$\Psi \colon \Lambda = \Z[e_1,e_2,\ldots] \to A$$ that we use to view Schur functions as elements in $A$. Recall that we often implicitly view elements of $\Lambda$ as elements of $A$ via $\Psi$. 
 We can then restate Theorem~\ref{thm:main} as follows. 

\begin{theorem}\label{thm:mainmod}
Let $A = A^0 \oplus \dotsb \oplus A^n$ be an algebra with the K\"{a}hler package with respect to a cone $\mathcal{K}_A$. For some $r \ge 1$, choose classes $c_i \in A^i$ for $i = 1, \dotsc, r$. Let $B = A[\zeta]/(\zeta^r - c_1 \zeta^{r-1} + \dotsb + (-1)^r c_r)$. Suppose that $B$ has the K\"{a}hler package with respect to $\mathcal{K}_B$. 
	Let $0 \le k \le n/2$ be a nonnegative integer. 
	Let $\lambda$ be a partition with $|\lambda| = n - 2k$, and 
	consider an element $a \in A^k$. Fix $h \in \K_A$ and consider  the rectangular decomposition of $\lambda$ with corresponding rectangles $R_1,\ldots,R_t$ as defined in \eqref{eq:rectangledef}.
	Suppose that 
	\begin{enumerate}
		\item\label{i:primitiveinsidemod} For any partition $\mu$ with $\lambda \subseteq \mu  \subseteq \lambda \cup R_1 \cup \cdots \cup R_t$ 
		\begin{equation*}
			a  \left( h s_\mu +   L_\lambda(s_\mu) \right) = 0 \text{ in }A, \textrm{ and},		
		\end{equation*}
		
		\item\label{i:primitiveoutsidemod} 
		 $a \cdot E^{(\le m)}_\lambda = 0 \text{ in }A$. 
	\end{enumerate}
	Then 
	\[
	(-1)^k\deg_A(a^2 s_\lambda) \ge 0,
	\]
	with equality if and only if 
	$a s_\mu = 0 \in A$ for all partitions $\mu$ with $\lambda \subseteq \mu$. 	
\end{theorem}

Let $m = \ell(\lambda)$. If $\ell(\lambda) > r$, then
$s_\mu = 0$ for all $\mu \supseteq \lambda$,
 so we may assume that $\ell(\lambda) \le r$.  Let $C$ be the $m$-fold multiprojective bundle ring. For example, if $\lambda = \emptyset$, then $m = 0$ and $C = A$. 
By Theorem~\ref{thm:multiproj}, $C$ has the K\"{a}hler package with respect to $\mathcal{K}_C$. 
Recall that  $\lambda$ has a rectangular decomposition
$\lambda=\widehat{\lambda}_1+ \cdots +\widehat{\lambda}_t$  with $\ell(\widehat{\lambda}_i) = m_i$.
Let $\overline{\lambda} = \widehat{\lambda}_2 + \cdots + \widehat{\lambda}_t$. For example, if $t = 1$, then $\overline{\lambda} =  \emptyset$. 
We will argue by induction on $t$, the number of parts in the rectangular decomposition of $\lambda$. Observe that when $t = 0$, $\lambda = \emptyset$, $k = n/2$, and Theorem~\ref{thm:mainmod} states that if $a h = 0$, then $(-1)^{n/2} \deg_A(a^2) \ge 0$, with equality if and only if $a = 0$. This follows because $A$ has the K\"ahler package with respect to $\mathcal{K}_A$. Assume that $t > 0$, i.e., we assume that $\lambda$ is nonempty. 

Recall that $\Delta = \prod_{1 \le i < j \le m} (\zeta_i - \zeta_j)$ is the Vandermonde determinant in $C^{\binom{m}{2}}$.
We may assume by induction that Theorem~\ref{thm:mainmod} holds for $\overline{\lambda}$. Let $k_1 = |\widehat{\lambda}_1|/m_1$ be the number of columns in $\widehat{\lambda}_1$. In what follows, we view the Young diagram of $\overline{\lambda}$ as the subdiagram of the Young diagram of $\lambda$ consisting of $\widehat{\lambda}_2,\ldots,\widehat{\lambda}_t$, or, equivalently, as the usual Young diagram of $\overline{\lambda}$ shifted to the right by $k_1$ boxes. In this way, $R_2,\ldots,R_t$ are the rectangles associated to the rectangular decomposition of $\overline{\lambda}$. If $\overline{\lambda} \subseteq \overline{\mu}$ for some partition $\overline{\mu}$, then we view the Young diagram of $\overline{\mu}$ to be the usual Young diagram shifted to the right by  $k_1$ boxes, so that it naturally contains the Young diagram of $\overline{\lambda}$. 
By Proposition~\ref{prop:pushforward} and Example~\ref{ex:rectangle}, for any $a, b \in A$, we have
\begin{align*}
 (-1)^{m(m-1)/2}m! \deg_A(ab s_\lambda) &= \deg_C( (a\Delta)(b \Delta) (\zeta_1 \cdots \zeta_m)^{r - m} s_\lambda(\zeta_1,\ldots,\zeta_m))	
\\
 &= \deg_C( (a\Delta)(b \Delta) (\zeta_1 \cdots \zeta_m)^{r - m}
 s_{\widehat{\lambda}_1}(\zeta_1,\ldots,\zeta_m) s_{\overline{\lambda}}(\zeta_1,\ldots,\zeta_m))	
\\
&= \deg_C( (a\Delta)(b \Delta) (\zeta_1 \cdots \zeta_m)^{r - m + k_1}s_{\overline{\lambda}}(\zeta_1,\ldots,\zeta_m))	
\\
&= \deg_{D}( (a\Delta)(b \Delta)s_{\overline{\lambda}}(\zeta_1,\ldots,\zeta_m)),	
\end{align*}
where $D = C/\ann((\zeta_1 \cdots \zeta_m)^{r - m + k_1})$. Since each $\zeta_i$ lies in $\overline{\mathcal{K}}_C$ and $C$ has the K\"{a}hler package with respect to $\mathcal{K}_C$, Proposition~\ref{prop:descent} implies that $D$ has the K\"ahler package with respect to the image $\K_{D}$ of $\K_C$. 
With a slight abuse of notation, let $a,h,\zeta_i$ denote the images of $a,h,\zeta_i$ respectively in $D$. Because each $\zeta_i$ is contained in $\overline{\mathcal{K}}_{D}$, Proposition~\ref{prop:nefdirectsum} implies that 
$D[y]/((y - \zeta_1) \cdots (y - \zeta_m))$ has the  K\"ahler package with respect to $\K_{D} + \R_{> 0} y$. 
We have $h + \zeta_1 + \cdots + \zeta_m = h + s_1(\zeta_1,\ldots,\zeta_m) \in \K_C$. By induction, we may apply Theorem~\ref{thm:mainmod}, replacing $A,B,$ and $\lambda$ with  $D$, $D[y]/((y - \zeta_1) \cdots (y - \zeta_m))$, and $\overline{\lambda}$, respectively. 

Suppose that 
 for any partition $\overline{\mu}$ with $\overline{\lambda} \subseteq \overline{\mu}$ and $m_2 = \ell(\overline{\lambda}) = \ell(\overline{\mu})$ either
\begin{enumerate}[label=(\Alph*),ref=\Alph*]
	
	\item\label{i:primitiveinsideproof} $\overline{\mu} \subseteq \overline{\lambda} \cup R_2 \cup \cdots \cup R_t$  and 
	\begin{equation*}
		a \Delta  \left( (h + s_1(\zeta_1,\ldots,\zeta_m)) s_{\overline{\mu}}(\zeta_1,\ldots,\zeta_m) +   \sum_{i = 2}^t (i - 1) \sum_{ \substack{ \overline{\nu} \supseteq \overline{\mu} \\ |\overline{\nu}| = |\overline{\mu}| + 1 \\ \overline{\nu} \smallsetminus \overline{\mu} \in R_i}} s_{\overline{\nu}}(\zeta_1,\ldots,\zeta_m) \right) = 0 \text{ in } D, \textrm{ or},
	\end{equation*}
	
	\item\label{i:primitiveoutsideproof} $\overline{\mu} \nsubseteq \overline{\lambda} \cup R_2 \cup \cdots \cup R_t$ and
	\[
	a \Delta s_{\overline{\mu}}(\zeta_1,\ldots,\zeta_m) = 0 \text{ in } D. 
	\] 
\end{enumerate}
For example,  when $t = 1$, we have $\overline{\lambda} = \overline{\mu} = \emptyset$,  \eqref{i:primitiveinsideproof} is the condition that
$a \Delta (h + s_1(\zeta_1,\ldots,\zeta_m)) = 0$ in $D$, and \eqref{i:primitiveoutsideproof} is vacuous. 
Then, by induction, 
\[
(-1)^{k + m(m - 1)/2}\deg_{D}((a \Delta)^2 s_{\overline{\lambda}}(\zeta_1,\ldots,\zeta_m)) \ge 0,
\]
with equality if and only if 
$a \Delta s_{\overline{\mu}} = 0$ in $D$ for all partitions $\overline{\mu}$ with $\overline{\lambda} \subseteq \overline{\mu}$.
In particular, we deduce that
\begin{equation*}
 (-1)^k\deg_A(a^2 s_\lambda) 
= 	(-1)^k (-1)^{m(m-1)/2}(1/m!)\deg_{D}( (a\Delta)^2 s_{\overline{\lambda}}(\zeta_1,\ldots,\zeta_m)) \ge 0.
\end{equation*}

Therefore, to prove that $(-1)^k\deg_A(a^2 s_\lambda) \ge 0$, we need to show that conditions \eqref{i:primitiveinsidemod} and \eqref{i:primitiveoutsidemod} in the statement of Theorem~\ref{thm:mainmod} imply that  conditions \eqref{i:primitiveinsideproof} and \eqref{i:primitiveoutsideproof} above hold.  

Consider a partition $\overline{\mu}$ with $\overline{\lambda} \subseteq \overline{\mu}$ and $m_2 = \ell(\overline{\lambda}) = \ell(\overline{\mu})$. 
Let $\mu = \widehat{\lambda}_1 + \overline{\mu}$. Then $\lambda = \widehat{\lambda}_1 + \overline{\lambda} \subseteq \widehat{\lambda}_1 + \overline{\mu} = \mu$ and $\ell(\mu) = \ell(\lambda) = m$. 
First assume that  $\overline{\mu} \subseteq \overline{\lambda} \cup R_2 \cup \cdots \cup R_t$. We need to show that condition \eqref{i:primitiveinsideproof} above holds. That is, we need the following equality in $C$:
\[
a \Delta (\zeta_1 \cdots \zeta_m)^{r - m + k_1} \left( (h + s_1(\zeta_1,\ldots,\zeta_m)) s_{\overline{\mu}}(\zeta_1,\ldots,\zeta_m) +   \sum_{i = 2}^t (i - 1) \sum_{ \substack{ \overline{\nu} \supseteq \overline{\mu}  \\ |\overline{\nu}| = |\overline{\mu}| + 1 \\ \overline{\nu} \smallsetminus \overline{\mu} \in R_i}} s_{\overline{\nu}}(\zeta_1,\ldots,\zeta_m) \right) = 0. 
\]
Equivalently, since $s_{\widehat{\lambda}_1}(\zeta_1,\ldots,\zeta_m) = (\zeta_1 \cdots \zeta_m)^{k_1}$ and using Example~\ref{ex:rectangle}, this simplifies to
\[
a \Delta (\zeta_1 \cdots \zeta_m)^{r - m} \left( (h + s_1(\zeta_1,\ldots,\zeta_m)) s_{\mu}(\zeta_1,\ldots,\zeta_m) +   \sum_{i = 2}^t (i - 1) \sum_{ \substack{ \nu \supseteq \mu \\ |\nu| = |\mu| + 1 \\ \nu \smallsetminus \mu \in R_i}} s_{\nu}(\zeta_1,\ldots,\zeta_m) \right) = 0.
\]
Using Pieri's rule, we compute $$s_1(\zeta_1,\ldots,\zeta_m)s_{\mu}(\zeta_1,\ldots,\zeta_m) = (s_1 \star_m s_{\mu})(\zeta_1,\ldots,\zeta_m) = \sum_{ \substack{ \nu \supseteq \mu \\ |\nu| = |\mu| + 1 \\ \ell(\nu) = m }} s_{\nu}(\zeta_1,\ldots,\zeta_m) \text{ in } C.$$ 
Substituting this into the above equation gives:
\begin{equation*}
a \Delta (\zeta_1 \cdots \zeta_m)^{r - m} \left( h s_{\mu}(\zeta_1,\ldots,\zeta_m) +   \sum_{i = 1}^t i \sum_{ \substack{ \nu \supseteq \mu \\ |\nu| = |\mu| + 1 \\ \nu \smallsetminus \mu \in R_i}} s_{\nu}(\zeta_1,\ldots,\zeta_m) + \sum_{ \substack{ \nu \supseteq \mu \\ |\nu| = |\mu| + 1 \\ \ell(\nu) = m \\ \nu \smallsetminus \mu \nsubseteq R_1 \cup \cdots \cup R_t }} s_\nu(\zeta_1,\ldots,\zeta_m) \right) = 0.
\end{equation*}
By Lemma~\ref{lem:detectzeroupstairs}, 
this is equivalent to 
\begin{equation*}
a  \left( h
s_{\mu} +   L_\lambda(s_\mu) + r_\mu \right) \star_m s_{\nu'}  = 0 \text{ in } A,
\end{equation*}
for all partitions $\nu'$  with $\ell(\nu') \le m$, where 
$$r_\mu = \sum_{ \substack{ \nu \supseteq \mu \\ |\nu| = |\mu| + 1 \\ \ell(\nu) = m \\ \nu \smallsetminus \mu \nsubseteq R_1 \cup \cdots \cup R_t }} s_\nu \in E^{(\le m)}_\lambda.$$
Recall that condition \eqref{i:primitiveoutsidemod} of Theorem~\ref{thm:mainmod} states that $a$ kills the ideal $E^{(\le m)}_\lambda$ in $\Lambda^{(\le m)}$. By Lemma~\ref{lem:commute}, we need to show that 
$$
a  \left( h
(s_{\mu} \star_m s_{\nu'}) +   L_\lambda(s_\mu \star_m s_{\nu'}) \right)   = 0 \text{ in } A,
$$ 
for all partitions $\nu'$  with $\ell(\nu') \le m$. By \eqref{i:primitiveinsidemod} and \eqref{i:primitiveoutsidemod} of Theorem~\ref{thm:mainmod}, for any partition $\mu'$ with $\lambda \subseteq \mu'$  and $\ell(\mu') = m$, $a  \left( h s_{\mu'} +   L_\lambda(s_{\mu'}) \right) = 0 \text{ in } A$, and the above equality follows.

Next assume that  $\overline{\mu} \nsubseteq \overline{\lambda} \cup R_2 \cup \cdots \cup R_t$. Recall that $\mu = \widehat{\lambda}_1 + \overline{\mu}$. Then $\mu \nsubseteq \lambda \cup R_1 \cup R_2 \cup \cdots \cup R_t$. We need to show that condition \eqref{i:primitiveoutsideproof} above holds. That is,
\[
a \Delta (\zeta_1 \cdots \zeta_m)^{r - m + k_1} s_{\overline{\mu}}(\zeta_1,\ldots,\zeta_m) = 0 \text{ in } C. 
\] 
Equivalently, since $s_{\widehat{\lambda}_1}(\zeta_1,\ldots,\zeta_m) = (\zeta_1 \cdots \zeta_m)^{k_1}$ and using Example~\ref{ex:rectangle}, this simplifies to
\begin{equation*}
	a \Delta (\zeta_1 \cdots \zeta_m)^{r - m} s_{\mu}(\zeta_1,\ldots,\zeta_m) = 0 \text{ in } C. 
\end{equation*} 
By Lemma~\ref{lem:detectzeroupstairs}, 
this is equivalent to 
\begin{equation*}
	a  (s_{\mu} \star_m s_{\nu}) = 0  \text{ in } A,
\end{equation*}
for all partitions $\nu$  with $\ell(\nu) \le m$. Then $s_{\mu}$ lies in the ideal $E_\lambda^{(\le m)}$ and the equality follows from condition \eqref{i:primitiveoutsidemod} of Theorem~\ref{thm:mainmod}.

To summarize, we have shown that conditions \eqref{i:primitiveinsidemod} and \eqref{i:primitiveoutsidemod} in Theorem~\ref{thm:mainmod} imply that conditions \eqref{i:primitiveinsideproof} and \eqref{i:primitiveoutsideproof} above hold, and the latter conditions imply that $(-1)^k\deg_A(a^2 s_\lambda) \ge 0$ with 
equality if and only if 
$a \Delta s_{\overline{\mu}} = 0$ in $D$ for all partitions $\overline{\mu}$ with $\overline{\lambda} \subseteq \overline{\mu}$.

It remains to consider the equality condition. 
Suppose that 	$a s_\mu = 0$ in $A$ for all partitions $\mu$ with $\lambda \subseteq \mu$. Then conditions \eqref{i:primitiveinsidemod} and \eqref{i:primitiveoutsidemod} in Theorem~\ref{thm:mainmod} are satisfied. Setting $\mu = \lambda$ implies that $a s_\lambda = 0$, and so $\deg_A(a^2 s_\lambda) = 0$. Conversely, suppose that conditions \eqref{i:primitiveinsidemod} and \eqref{i:primitiveoutsidemod} in Theorem~\ref{thm:mainmod} are satisfied and $\deg_A(a^2 s_\lambda) = 0$. 
Then 
 $a \Delta s_{\overline{\mu}} = 0 \text{ in } D$ for all partitions $\overline{\mu}$ with $\overline{\lambda} \subseteq \overline{\mu}$. 
Equivalently, 
\[
	a \Delta (\zeta_1 \cdots \zeta_m)^{r - m} s_{\widehat{\lambda}_1 + \overline{\mu}}(\zeta_1,\ldots,\zeta_m) = 0 \text{ in } C. 
\]
for all $\overline{\mu}$ with $\overline{\lambda} \subseteq \overline{\mu}$  and  $\ell(\overline{\mu}) = m_2$. By Lemma~\ref{lem:detectzeroupstairs}, 
this is equivalent to 
\begin{equation}\label{eq:equalityconditionpusheddown}
	a  (s_{\widehat{\lambda}_1 + \overline{\mu}} \star_m s_{\nu}) 
	= a (s_{\widehat{\lambda}_1} \star_m (s_{\overline{\mu}} \star_m s_{\nu})) 
	= 0 \text{ in } A,
\end{equation}
for all $\overline{\mu}$ with $\overline{\lambda} \subseteq \overline{\mu}$  and  $\ell(\overline{\mu}) = m_2$, and for all partitions $\nu$  with $\ell(\nu) \le m$.

By Lemma~\ref{lem:idealgenerated}, the ideal $\Lambda_{\overline{\lambda}}^{(\le m)} \subseteq \Lambda^{(\le m)}$ is generated by $\Lambda_{\overline{\lambda}}^{(\le m_2)}$. Hence $\Lambda_{\overline{\lambda}}^{(\le m)}$ is generated as an abelian group by
$\{ s_{\overline{\mu}} \star_m s_{\nu} : \overline{\lambda} \subseteq \overline{\mu}, \, \ell(\overline{\mu}) = m_2, \, \ell(\nu) \le m \}$. 
Then \eqref{eq:equalityconditionpusheddown} implies that $a \cdot \Lambda_{\lambda}^{(\le m)} = 0$ in $A$. 
By Lemma~\ref{lem:idealgenerated},  the ideal 	$\Lambda_\lambda \subseteq \Lambda$ is generated by $\Lambda_\lambda^{(\le m)}$. It follows that 
$a \cdot \Lambda_\lambda = 0$. This completes the proof.

\subsection{Strict positivity}\label{sec:strict}

In this section, we prove a version of Theorem~\ref{thm:main} which holds when $B$ has the K\"{a}hler package with respect to $\mathcal{K}_B$ and $\zeta$ satisfies the hard Lefschetz property in $B$. For example, this holds when $A$ is the ring of real $(p, p)$ classes on a smooth complex projective variety and $B$ is the ring of real $(p, p)$ classes on the projectivization of an ample vector bundle. 

We can construct more examples by ``twisting.'' Given a vector bundle ${E}$ and a line bundle $\mathcal{L}$ on a variety $X$, the projectivization $\mathbb{P}_X({E} \otimes \mathcal{L})$ is isomorphic to the projectivization $\mathbb{P}_X({E})$, but the relative $\mathcal{O}(1)$ is replaced by the tensor product of the original relative $\mathcal{O}(1)$ with the pullback of $\mathcal{L}$. It will be convenient to consider an analogue of this operation for projective bundle rings. 
Given $\delta \in A^1$, for each $i$ set $c_i\langle \delta \rangle = \sum_{j=0}^{i} \binom{r - i + j}{j} c_{i-j} \delta^j$. There is an isomorphism of $A$-algebras
$$\frac{A[\zeta]}{(\zeta^r - c_1\langle \delta \rangle \zeta^{r-1} + \dotsb + (-1)^r c_r\langle \delta \rangle)} \to B, \quad \zeta \mapsto \zeta + \delta.$$
 Suppose that $B$ satisfies the K\"{a}hler package with respect to $\mathcal{K}_B$ and choose some $h \in \mathcal{K}_A$. Then $\zeta + h$ satisfies the hard Lefschetz property in $B$. 
In particular, the projective bundle ring defined using $c_1\langle h \rangle, \dotsc, c_r\langle h \rangle$ has the K\"{a}hler package with respect to
 positive linear combinations of the images of elements of $\mathcal{K}_A$ and $\zeta$,  
 and $\zeta$ satisfies the hard Lefschetz property. 
Write $s_{\lambda}\langle \delta \rangle$ for the Schur classes defined using $c_1\langle \delta \rangle, \dotsc, c_r \langle \delta \rangle$. These Schur classes can be related to the Schur classes defined using $c_1, \dotsc, c_r$, as follows. For each partition $\lambda$ with at most $r$ parts, we can write
$$s_{\lambda}(x_1 + y, x_2+ y, \dotsc, x_r + y) = s_{\lambda}(x_1, \dotsc, x_r) + y s_{\lambda}^{(1)}(x_1, \dotsc, x_r) + y^2 s_{\lambda}^{(2)}(x_1, \dotsc, x_r) + \dotsb.$$
See \cite[Section 2.5]{RT1}. Each $s_{\lambda}^{(i)}$ is a symmetric polynomial, so we can write $s_{\lambda}^{(i)} = \sum_{|\mu| = |\lambda| - i} a_{\mu} s_{\mu}(x_1, \dotsc, x_r)$ for some integers $a_{\mu}$. Then we have
\begin{equation}\label{eq:twist}
s_{\lambda} \langle \delta \rangle = s_{\lambda} +  \delta\sum_{|\mu| = |\lambda| - 1 } a_{\mu} s_{\mu} +  \delta^2\sum_{|\mu| = |\lambda| - 2} a_{\mu} s_{\mu} + \dotsb.
\end{equation}
Note that the numbers $a_{\mu}$ depend on $r$ and not only on $\lambda$. Note also that $a_{\emptyset} > 0$; it is equal to $s_{\lambda}(1, \dotsc, 1)$. 
In \cite{CorteelKimBounded}, Corteel and Kim showed that each $a_{\mu}$ is nonnegative. This can also be seen using the converse to the fact that degrees of Schur classes of nef vector bundles are nonnegative: if $E$ is a nef vector bundle on $X$, then the vector bundle $E \boxtimes \mathcal{O}(1)$ on $X \times \mathbb{P}^i$ is nef. For a partition $\lambda$ with $|\lambda| = \dim X + i$, we have
$$\deg_{X \times \mathbb{P}^i}(s_{\lambda}(E \boxtimes \mathcal{O}(1))) = \deg_X(s_{\lambda}^{(i)}(E)).$$
As the left-hand side of the above equation is nonnegative for all nef vector bundles $E$, \cite[Theorem 1]{FultonLaz} implies that the expansion of $s_{\lambda}^{(i)}$ in terms of Schur classes is nonnegative.

\begin{proof}[Proof of Theorem~\ref{thm:strict}]
Choose some $h \in \mathcal{K}_A$. 
Because $\zeta$ satisfies the hard Lefschetz property, so does $\zeta - \varepsilon h$ for sufficiently small $\varepsilon > 0$. Proposition~\ref{prop:convex} implies that there is an open convex cone $\tilde{\mathcal{K}} \subseteq B^1$ containing $\mathcal{K}_B$ and $\zeta$ on which $B$ has the K\"{a}hler package. This implies that the ring 
$$B_{\varepsilon} = \frac{A[\zeta]}{(\zeta^r - c_1\langle - \varepsilon h \rangle \zeta^{r-1} + \dotsb + (-1)^r c_r\langle - \varepsilon h \rangle)}$$
has the K\"{a}hler package with respect to $\pi^*(\mathcal{K}_A) + \mathbb{R}_{>0} \zeta$. By Theorem~\ref{thm:main}, $\deg_A(s_{\lambda} \langle - \varepsilon h \rangle) \ge 0$. Also, Proposition~\ref{prop:descent} and Theorem~\ref{thm:main} imply that for any partition $\mu$ with $|\mu| \le n$, we have $\deg_A(s_{\mu}\langle -\varepsilon h \rangle h^{n - |\mu|}) \ge 0$. By \eqref{eq:twist} we have
$$s_{\lambda} = s_{\lambda} \langle - \varepsilon h \rangle + \varepsilon h \sum_{|\mu| = n-1} a_{\mu} s_{\mu} \langle - \varepsilon h \rangle + \dotsb + a_{\emptyset} \varepsilon^n h^n. $$
Applying $\deg_A$ to this equation and using that $\deg_A(h^n) > 0$ implies the result. 
\end{proof}

\section{Variants}\label{sec:variants}

In this section, we give several variants and consequences of the main results. As mentioned in the introduction, Theorem~\ref{thm:main} generalizes the results of Ross and Toma \cites{RT1,RT2,RT3} to a purely algebraic setting. In \cite[Section 10]{RT2}, Ross and Toma give a number of inequalities which follow from their results.  The proofs of these inequalities can all be adapted to the purely algebraic setting of Theorem~\ref{thm:main}. In particular, they hold in the setting of \cite{LarsonPartida}.

\subsection{Schur positivity for Lefschetz modules}\label{ss:Lefschetz}

We begin by describing a version of Theorem~\ref{thm:main} for modules. The motivation for considering this version is that the intersection cohomology of a complex projective variety satisfies a version of the K\"{a}hler package \cite{Saito88}. The intersection cohomology of a complex projective variety does not have a natural ring structure, but rather it is a module over the cohomology ring. Let $A$ be a finite-dimensional commutative graded $\mathbb{R}$-algebra, and let $M = M^0 \oplus \dotsb \oplus M^n$ be a finite-dimensional graded $A$-module. Let $\mathcal{K}_A$ be a nonempty open convex cone in $A^1$, and let $\mathcal{Q} \colon M \times M \to \mathbb{R}$ be a symmetric bilinear form on $M$ which respects the $A$-action, i.e.,
$$\mathcal{Q}(zx, y) = \mathcal{Q}(x, zy) \text{ for all }z \in A \text{ and }x, y \in M.$$
We say that $(M, \mathcal{Q})$ is a \emph{Lefschetz module} of degree $n$ over $(A, \mathcal{K}_A)$ if the following properties hold:
\begin{enumerate}
\item The pairing $M^i \times M^j \to \mathbb{R}$ induced by $\mathcal{Q}$ is nondegenerate if $i + j = n$, and it is $0$ otherwise. 
\item For all $h \in \mathcal{K}_A$ and $i < n/2$, multiplication by $h^{n - 2i}$ induces an isomorphism from $M^i$ to $M^{n - i}$. 
\item For all $h \in \mathcal{K}_A$ and $i \le n/2$, the symmetric bilinear form on $M^i$ given by $(x, y) \mapsto (-1)^i\mathcal{Q}(x, h^{n - 2i}y)$ is positive definite when restricted to the kernel of multiplication by $h^{n - 2i + 1}$. 
\end{enumerate}

See \cite{LefschetzModule}. Given classes $c_1, \dotsc, c_r$, with $c_i \in A^i$, define classes $s_1, s_2, \dotsc$ via the formula \eqref{eq:segre}. Set $s_0 = 1$ and $s_{j} = 0$ for $j < 0$. 
Let $B = A[\zeta]/(\zeta^r - c_1 \zeta^{r-1} + \dotsb + (-1)^r c_r)$ with natural inclusion $\pi^* \colon A \to B$, and consider the $B$-module $\tilde{M}$ which is given as a graded $A$-module by $M \oplus M[-1] \oplus \dotsb \oplus M[-r+1]$, where the action of $\zeta$ is given as follows. For $0 \le i \le r-2$, multiplication by $\zeta$ sends $M[-i]$ onto $M[-i - 1]$ via the natural map. We can therefore write each element of $M[-i]$ as $\zeta^i x$ for a unique choice of $x \in M$. The action of $\zeta$ on $M[-r+1]$ is given by 
$$\zeta \cdot (\zeta^{r-1} x) = \zeta^{r-1} c_1 x - \zeta^{r-2} c_2 x + \dotsb + (-1)^{r-1} c_r x.$$
When $M = A$, this construction recovers the projective bundle ring. Set $\mathcal{K}_B = \pi^*(\mathcal{K}_A) + \mathbb{R}_{>0} \zeta$. 
There is a bilinear form $\tilde{\mathcal{Q}}$ on $\tilde{M}$ given by 
$$\tilde{\mathcal{Q}}(\zeta^i x, \zeta^j y) = \mathcal{Q}(s_{i + j - r + 1}x, y)$$
for $x, y \in M$. The natural inclusion $\pi^* \colon M \to \tilde{M}$ is adjoint to the pushforward map $\pi_* \colon \tilde{M} \to M$, which sends $\zeta^{r-1} x$ to $x$ and $\zeta^j x$ to $0$ for all $x \in M$ and $j < r-1$. 
All of the arguments in the proof of Theorem~\ref{thm:main} can be adapted to prove the following theorem.

\begin{theorem}\label{thm:moduleversion}
Let $(M, \mathcal{Q})$ be a Lefschetz module of degree $n$ over $(A, \mathcal{K}_A)$. 
For some $r \ge 1$, choose classes $c_i \in A^i$ for $i = 1, \dotsc, r$. Let $B = A[\zeta]/(\zeta^r - c_1 \zeta^{r-1} + \dotsb + (-1)^r c_r)$, and let $\tilde{M}$ be the graded $B$-module described above. Suppose that $(\tilde{M}, \tilde{\mathcal{Q}})$ is a Lefschetz module over $(B, \mathcal{K}_B)$.  
Let $k \le n/2$ be a nonnegative integer. 
Let $\lambda$ be a partition with $|\lambda| = n - 2k$, and 
consider an element $m \in M^k$. Fix $h \in \K_A$ and consider  the rectangular decomposition of $\lambda$ with corresponding rectangles $R_1,\ldots,R_t$ as defined in \eqref{eq:rectangledef}.
Suppose that for any partition $\mu$ with $\lambda \subseteq \mu$ and $\ell(\lambda) = \ell(\mu)$, either
 \begin{enumerate}
 	
 	\item$\mu \subseteq \lambda \cup R_1 \cup \cdots \cup R_t$  and 
 	\begin{equation*}\label{eq:primitive}
 		 	 \left( h s_\mu +   \sum_{i = 1}^t i \sum_{ \substack{ \mu \subseteq \nu \\ |\nu| = |\mu| + 1 \\ \nu \smallsetminus \mu \in R_i}} s_\nu \right) m = 0, \textrm{ or},
 	\end{equation*}

 	\item$\mu \nsubseteq \lambda \cup R_1 \cup \cdots \cup R_t$ and  $s_\mu m = 0$. 
 \end{enumerate}
 Then 
 \[
 (-1)^k\mathcal{Q}(m, s_{\lambda} m) \ge 0,
 \]
 with equality if and only if $ s_\mu m = 0$ in $M$ for all partitions $\mu$ with $\lambda \subseteq \mu$. 
\end{theorem}

For example, Theorem~\ref{thm:moduleversion} can be applied in the following setting. Let $X$ be a complex projective variety, and let $H$ be the subalgebra of $H^*(X; \mathbb{R})$ generated by the Chern classes of vector bundles. The work of Saito \cite{Saito88} endows the intersection cohomology $IH(X)$ with a pure Hodge structure, and shows that $\bigoplus_p IH^{p,p}(X)$, the module of real $(p,p)$ classes, is a Lefschetz module over $H$ with respect to the ample cone. If $E$ is a nef vector bundle on $X$, then the ring $B$ considered in Theorem~\ref{thm:moduleversion} is equal to the subalgebra of $H^*(\mathbb{P}_X(E); \mathbb{R})$ generated by the Chern classes of vector bundles, and the module $\tilde{M}$ is equal to the module of real $(p, p)$ classes in $IH(\mathbb{P}_X(E))$.  Because $E$ is nef, $\tilde{M}$ is a Lefschetz module over $(B, \mathcal{K}_B)$. 

The deduction of Theorem~\ref{thm:strict} from Theorem~\ref{thm:main} in the setting of Lefschetz modules also adapts to give the following generalization of Theorem~\ref{thm:strict}. 

\begin{theorem}\label{thm:modulestrict}
Let $(M, \mathcal{Q})$ be a Lefschetz module of degree $n$ over $(A, \mathcal{K}_A)$. 
For some $r \ge 1$, choose classes $c_i \in A^i$ for $i = 1, \dotsc, r$. Let $B = A[\zeta]/(\zeta^r - c_1 \zeta^{r-1} + \dotsb + (-1)^r c_r)$, and let $\tilde{M}$ be the graded $B$-module described above. Suppose that $(\tilde{M}, \tilde{\mathcal{Q}})$ is a Lefschetz module over $(B, \mathcal{K}_B)$, and that $\zeta$ satisfies the hard Lefschetz property in $\tilde{M}$.  
Then for any partition $\lambda$ with $|\lambda| = n$ and $\ell(\lambda) \le r$, the symmetric bilinear form on $M^0$ given by $(m_1, m_2) \mapsto \mathcal{Q}(m_1, s_{\lambda} m_2)$ is positive definite. 
\end{theorem}

\subsection{Schur positivity for complex Lefschetz modules}\label{ss:complexLefschetz}

There is also a version of the main theorems for \emph{complex Lefschetz modules}. This is useful to handle intersection cohomology rings of varieties which have nontrivial Hodge structures. Let $A$ be a finite-dimensional commutative graded $\mathbb{R}$-algebra with a nonempty open convex cone $\mathcal{K}_A$ in $A^1$, and set $A_{\mathbb{C}} = A \otimes_{\mathbb{R}} \mathbb{C}$.  Let $M = M^0 \oplus \dotsb \oplus M^n$ be a finite-dimensional graded $A_{\mathbb{C}}$-module, and let $\mathcal{Q} \colon M \times M \to \mathbb{C}$ be a Hermitian form that respects the action of $A_{\mathbb{C}}$, i.e., 
$$\mathcal{Q}(zx, y) = \mathcal{Q}(x, \overline{z}y) \text{ for all }z \in A_{\mathbb{C}} \text{ and }x, y \in M.$$
We say that $(M, \mathcal{Q})$ is a complex Lefschetz module of degree $n$ over $(A, \mathcal{K}_A)$ if the following properties hold:
\begin{enumerate}
\item The pairing $M^i \times M^j \to \mathbb{C}$ induced by $\mathcal{Q}$ is nondegenerate if $i + j = n$, and it is $0$ otherwise. 
\item For all $h \in \mathcal{K}_A$ and $i < n/2$, multiplication by $h^{n - 2i}$ induces an isomorphism from $M^i$ to $M^{n - i}$. 
\item For all $h \in \mathcal{K}_A$ and $i \le n/2$, the Hermitian form on $M^i$ given by $(x, y) \mapsto (-1)^i\mathcal{Q}(x, h^{n - 2i}y)$ is positive definite when restricted to the kernel of multiplication by $h^{n - 2i + 1}$. 
\end{enumerate}
See \cite[Section 3.4.1]{LefschetzModule}. Given classes $c_1, \dotsc, c_r$, with $c_i \in A^i$, let $B = A[\zeta]/(\zeta^r - c_1 \zeta^{r-1} + \dotsb + (-1)^r c_r)$. We can define a $B_{\mathbb{C}}$-module $\tilde{M}$ and a Hermitian form $\tilde{\mathcal{Q}}$ on it which respects the action of $A_{\mathbb{C}}$ using the same formulas as 
in Section~\ref{ss:Lefschetz}. 
The proof of Theorem~\ref{thm:main} can be adapted to prove the following result. 
\begin{theorem}\label{thm:complexmodule}
Let $(M, \mathcal{Q})$ be a complex Lefschetz module of degree $n$ over $(A, \mathcal{K}_A)$. 
For some $r \ge 1$, choose classes $c_i \in A^i$ for $i = 1, \dotsc, r$. Let $B = A[\zeta]/(\zeta^r - c_1 \zeta^{r-1} + \dotsb + (-1)^r c_r)$, and let $\tilde{M}$ be the graded $B_{\mathbb{C}}$-module described above. Suppose that $(\tilde{M}, \tilde{\mathcal{Q}})$ is a complex Lefschetz module over $(B, \mathcal{K}_B)$.  
Let $k \le n/2$ be a nonnegative integer. 
Let $\lambda$ be a partition with $|\lambda| = n - 2k$, and 
consider an element $m \in M^k$. Fix $h \in \K_A$ and consider  the rectangular decomposition of $\lambda$ with corresponding rectangles $R_1,\ldots,R_t$ as defined in \eqref{eq:rectangledef}.
Suppose that for any partition $\mu$ with $\lambda \subseteq \mu$ and $\ell(\lambda) = \ell(\mu)$, either
 \begin{enumerate}
 	
 	\item$\mu \subseteq \lambda \cup R_1 \cup \cdots \cup R_t$  and 
 	\begin{equation*}\label{eq:primitive}
 		 	 \left( h s_\mu +   \sum_{i = 1}^t i \sum_{ \substack{ \mu \subseteq \nu \\ |\nu| = |\mu| + 1 \\ \nu \smallsetminus \mu \in R_i}} s_\nu \right) m = 0, \textrm{ or},
 	\end{equation*}

 	\item$\mu \nsubseteq \lambda \cup R_1 \cup \cdots \cup R_t$ and  $s_\mu m = 0$. 
 \end{enumerate}
 Then 
 \[
 (-1)^k\mathcal{Q}(m, s_{\lambda} m) \ge 0,
 \]
 with equality if and only if $ s_\mu m = 0$ in $M$ for all partitions $\mu$ with $\lambda \subseteq \mu$. 
\end{theorem}
One can similarly prove an analogue of Theorem~\ref{thm:modulestrict} for complex Lefschetz modules; see \cite[Example 3.11]{LefschetzModule}. The main case of interest for Theorem~\ref{thm:complexmodule} is the following. Let $X$ be a smooth complex projective variety, and let $H$ be the ring of real $(p, p)$ classes on $X$. For some nonnegative integer $k \le \dim X$, set
$$M = \bigoplus_{p \ge 0} H^{p+k, p}(X;\mathbb{C}).$$
Then $M$ is a complex Lefschetz module of degree $\dim X - k$ over $H$ with respect to the ample cone. If $E$ is a nef vector bundle on $X$, then the ring $B$ considered in Theorem~\ref{thm:complexmodule} is equal to the ring of real $(p, p)$ classes on $\mathbb{P}_X(E)$, and the module $\tilde{M}$ is equal to $\bigoplus_{p \ge 0} H^{p + k, p}(\mathbb{P}_X(E); \mathbb{C})$. Because $E$ is nef, $\tilde{M}$ is a complex Lefschetz module over $(B, \mathcal{K}_B)$.

\subsection{Multiple projective bundle rings}

There is a version of our main results for multiple projective bundle rings. This extends results in \cite[Section~3.2]{FultonLaz} and \cite[Section~9.2]{RT1}. The following results can also be extended to Lefschetz modules and complex Lefschetz modules as in Section~\ref{ss:Lefschetz} and Section~\ref{ss:complexLefschetz} respectively.

\begin{theorem}\label{thm:mainmultiple}
	Let $A = A^0 \oplus \dotsb \oplus A^n$ be an algebra with the K\"{a}hler package with respect to a cone $\mathcal{K}_A$. 
	For some $r_1,\ldots,r_q \ge 1$, choose classes 
	$c_{j,i} \in A^i$ for $1 \le i \le r_j$ and $1 \le j \le q$. 
	For $1 \le j \le q$, let $B_j = A[\zeta]/(\zeta^{r_j} - c_{j,1} \zeta^{r_j - 1} + \dotsb + (-1)^{r_j} c_{j, r_j})$, and suppose that  $B_j$ has the K\"{a}hler package with respect to $\mathcal{K}_{B_j}$. 
	Let $k \le n/2$ be a nonnegative integer. 
	Let $\lambda_1,\ldots,\lambda_q$ be partitions with $|\lambda_1| + \cdots + |\lambda_q| = n - 2k$, and  
	consider an element $a \in A^k$.
	Fix $h_1,\ldots,h_q \in \K_A$ and, for $1 \le j \le q$, consider  the rectangular decomposition of $\lambda_j$ with corresponding rectangles $R_{j,1},\ldots,R_{j,t_j}$ as defined in \eqref{eq:rectangledef}.
	Suppose that for any partitions $\mu_1,\ldots,\mu_q$ with $\lambda_j \subseteq \mu_j$ and $\ell(\lambda_j) = \ell(\mu_j)$ for $1 \le j \le q$,
	\begin{equation}\label{eq:multiplecondition}
		a \hat{s}_{\mu_1} \cdots \hat{s}_{\mu_q} = 0,
	\end{equation}
	where
	\[
	\hat{s}_{\mu_j} = \begin{cases}
		h_j s_{\mu_j} +   \sum_{i = 1}^{t_j} i \sum\limits_{ \substack{ \nu \supseteq \mu_j \\ |\nu| = |\mu_j| + 1 \\ \nu \smallsetminus \mu_j \in R_{j,i}}} s_\nu
		&\textrm{if } \mu_j \subseteq \lambda_j \cup R_{j,1} \cup \cdots \cup R_{j,t_j}, 
		\\
		s_{\mu_j}
		&\textrm{otherwise, } 
	\end{cases}
	\]
	for $1 \le j \le q$.
	Then 
	\[
	(-1)^k\deg_A(a^2 s_{\lambda_1}\cdots s_{\lambda_q}) \ge 0,
	\]
	with equality if and only if $a s_{\mu_1} \cdots s_{\mu_q} = 0$ in $A$  
	for any partitions $\mu_1,\ldots,\mu_q$ with $\lambda_j \subseteq \mu_j$  for $1 \le j \le q$.
\end{theorem}

The proof is via induction and is a slight adaptation of the proof of Theorem~\ref{thm:main}. In more detail,  assume that $\lambda_1$ is nonempty, and set $m = \ell(\lambda_1)$, $k_1 = |\widehat{\lambda}_{1,1}|/m$, and consider the $m$-fold projective bundle 
$$C = \frac{A[\zeta_1, \dotsc, \zeta_m]}{(\zeta_1^{r_1} - c_{1,1} \zeta_1^{r_1-1} + \dotsb + (-1)^{r_1} c_{1,r_1}, \dotsc, \zeta_m^{r_1} - c_{1,1} \zeta_{m}^{r_1-1} + \dotsb + (-1)^{r_1} c_{1,r_1})},$$
and $D = C/\ann((\zeta_1 \cdots \zeta_m)^{r_1 - m + k_1})$. Then Remark~\ref{rem:multiple} and Proposition~\ref{prop:descent} imply that
$B_j' = D[\zeta]/(\zeta^{r_j} - c_{j,1} \zeta^{r_j - 1} + \dotsb + (-1)^{r_j} c_{j, r_j})$ has the K\"{a}hler package with respect to $\mathcal{K}_{B_j'}$ for $2 \le j \le q$. Then we follow the steps of the proof of Theorem~\ref{thm:main},  often absorbing conditions involving $B_2,\ldots,B_q$ into $a$.
We have the following analogue of Example~\ref{ex:FL}. 

\begin{example}\label{ex:FLmultiple}
	Suppose that $k=0$, so $|\lambda_1| + \cdots + |\lambda_q| = n$.
	Then \eqref{eq:multiplecondition} is automatic because $A^s = 0$ for $s > n$. So if each $B_j$ has the K\"{a}hler package with respect to $\mathcal{K}_{B_j}$, then $\deg_A(s_{\lambda_1}\cdots s_{\lambda_q}) \ge 0$. 
\end{example}

We also have the following analogue of Theorem~\ref{thm:strict}, which generalizes \cite[Corollary 3.10]{FultonLaz}. 

\begin{theorem}\label{thm:strictmultiple}
	Let $A = A^0 \oplus \dotsb \oplus A^n$ be an algebra with the K\"{a}hler package with respect to a cone $\mathcal{K}_A$. 
	For some $r_1,\ldots,r_q \ge 1$, choose classes 
	$c_{j,i} \in A^i$ for $1 \le i \le r_j$ and $1 \le j \le q$. 
	For $1 \le j \le q$, let $B_j = A[\zeta]/(\zeta^{r_j} - c_{j,1} \zeta^{r_j - 1} + \dotsb + (-1)^{r_j} c_{j, r_j})$, and suppose that  $B_j$ has the K\"{a}hler package with respect to $\mathcal{K}_{B_j}$, and that 
	$\zeta$ satisfies the hard Lefschetz property in $B_j$. 
	Then for any partitions $\lambda_1,\ldots,\lambda_q$ with 
	$|\lambda_1| + \cdots + |\lambda_q| = n$ and $\ell(\lambda_j) \le r_j$ for $1 \le j \le q$, we have $\deg_A(s_{\lambda_1} \cdots s_{\lambda_q}) > 0$. 
	\end{theorem}

The proof follows just as  the proof of Theorem~\ref{thm:strict}, by twisting each $\zeta$ by $-\epsilon h$ for some $h \in \K_A$.

\subsection{Lorentzian polynomials}

Given an algebra $A = A^0 \oplus \dotsb \oplus A^n$ which satisfies Poincar\'{e} duality and is equipped with a nonempty open convex cone $\mathcal{K}_A \subseteq A^1$, we say that it satisfies the mixed Hodge--Riemann relations in degree at most $1$ if
\begin{enumerate}
\item For any $\ell_1, \dotsc, \ell_n$ in $\mathcal{K}_A$, we have $\deg_A(\ell_1 \dotsb \ell_n) > 0$. 
\item For any $\ell_1, \dotsc, \ell_{n-2}$ in $\mathcal{K}_A$, the bilinear form $A^1 \times A^1 \to \mathbb{R}$ given by $(x, y) \mapsto \deg_A(\ell_1 \dotsb \ell_{n-2} xy)$ is nondegenerate and has exactly one positive eigenvalue. 
\end{enumerate}

If $A$ has the K\"{a}hler package with respect to $\mathcal{K}_A$, then it follows from Proposition~\ref{prop:descent} 
that $A$ has the mixed Hodge--Riemann relations in degree at most $1$; see \cite{Cat08}.  Whether $A$ has the mixed Hodge--Riemann relations in degree at most $1$ can be understood in terms of the theory of \emph{Lorentzian polynomials on cones}, developed in \cite{LorentzianCone} after earlier work in \cite{BH}. Choose a spanning set $x_1, \dotsc, x_d$ for $A^1$. Then $A$ has the mixed Hodge--Riemann relations in degree at most $1$ if and only if the polynomial
$$(y_1, \dotsc, y_d) \mapsto \deg_A \left( \left(\sum_{i=1}^{d} y_i x_i \right)^n \right)$$
is Lorentzian with respect to the preimage of $\mathcal{K}_A$ under the linear map $\mathbb{R}^d \to {A}^1$ induced by the choice of $x_1, \dotsc, x_d$; see \cite[Section 2]{LorentzianCone} and \cite[Theorem 3.8]{MNY}. Importantly, it follows from these results that the second condition in the definition of the mixed Hodge--Riemann relations in degree at most $1$ can be relaxed to 
\begin{enumerate}
\item[(2')]  For any $\ell_1, \dotsc, \ell_{n-2}$ in $\mathcal{K}_A$, the bilinear form $A^1 \times A^1 \to \mathbb{R}$ given by $(x, y) \mapsto \deg_A(\ell_1 \dotsb \ell_{n-2}x y )$ has at most one positive eigenvalue. 
\end{enumerate}

\begin{proposition}\label{prop:annschur}
Suppose that $B$ has the K\"{a}hler package with respect to $\mathcal{K}_B$. Then for any partition $\lambda$, $A/\operatorname{ann}(s_{\lambda})$ has the mixed Hodge--Riemann relations in degree at most $1$ with respect to the image of $\mathcal{K}_A$. 
\end{proposition}

\begin{proof}
We may assume that $|\lambda| \le n$, else $A/\operatorname{ann}(s_{\lambda})$ is the zero ring. 
Choose some $h \in \mathcal{K}_A$. With the notation of Section~\ref{sec:strict}, we prove that, for any $\varepsilon > 0$, $A/\operatorname{ann}(s_{\lambda}\langle \varepsilon h \rangle)$ satisfies the mixed Hodge--Riemann relations in degree at most $1$. As the space of polynomials which are Lorentzian with respect to $\mathcal{K}_A$ is closed \cite[Remark 2.5]{LorentzianCone}, this implies the result. 

To check the first condition in the definition of the mixed Hodge--Riemann relations in degree at most $1$, we need to show that for any $\ell_1, \dotsc, \ell_{n - |\lambda|}$ in $\mathcal{K}_A$, we have 
$$\deg_A(\ell_1 \dotsb \ell_{n-|\lambda|} s_{\lambda} \langle \varepsilon h \rangle) = \deg_{A/\operatorname{ann}(\ell_1 \dotsb \ell_{n-|\lambda|})} (s_{\lambda} \langle \varepsilon h \rangle)> 0.$$
This follows from Proposition~\ref{prop:descent} and Theorem~\ref{thm:strict}. 
To check the second condition, note that Proposition~\ref{prop:descent} implies that Theorem~\ref{thm:main} can be applied to $A/\operatorname{ann}(\ell_1 \dotsb \ell_{n - |\lambda| - 2})$.
Example~\ref{ex:rossToma} then 
 gives that the form $A^1 \times A^1 \to \mathbb{R}$ given by $(a, b) \mapsto \deg_A(ab \ell_1 \dotsb \ell_{n - |\lambda| - 2} s_{\lambda} \langle \varepsilon h \rangle)$ has at most one positive eigenvalue. This form then descends to a form on the degree $1$ part of $A/\operatorname{ann}(s_{\lambda} \langle \varepsilon h \rangle)$ which has at most one positive eigenvalue. 
\end{proof}

\begin{example}\label{ex:denormalizedSchur}
Let $\lambda$ be a partition with at most $m$ parts whose largest part has size at most $n$. 
Let $A = \mathbb{R}[x_1, \dotsc, x_m]/(x_1^{n+1}, \dotsc, x_m^{n+1})$ and let $\K_A = \mathbb{R}_{>0} x_1 + \dotsb + \mathbb{R}_{>0} x_m$. Then $A$ has the K\"{a}hler package with respect to $\K_A$ (for example, by Proposition~\ref{prop:civanishexample}).
Let $c_1, \dotsc, c_m$ be classes defined by the property that
$$1 + c_1 + \dotsb + c_m = (1 + x_1) \dotsb (1 + x_m),$$
and $c_i \in A^i$. 
By Proposition~\ref{prop:nefdirectsum}, the corresponding projective bundle ring $B$ has the K\"{a}hler package with respect to $\mathcal{K}_B$. Proposition~\ref{prop:annschur} then implies that $A/\operatorname{ann}(s_{\lambda})$ has the mixed Hodge--Riemann relations in degree at most $1$, so the polynomial 
$$(y_1, \dotsc, y_m) \mapsto \deg_A\left( \left(y_1 x_1 + \dotsb + y_m x_m\right )^{nm - |\lambda|} s_{\lambda} \right)$$
is Lorentzian. As observed in \cite[Proof of Theorem 3]{HMMS}, this fact, applied to all partitions $\lambda$, implies that Schur polynomials are denormalized Lorentzian, in the sense of \cite{BH}. 
This result was originally proven in \cite{HMMS} and a second proof was given in \cite[Section 10.3]{RT2}. Unlike previous proofs of this fact, the above argument does not use algebraic geometry. However, all three arguments are based on establishing the mixed Hodge--Riemann relations in degree at most $1$ for $A/\operatorname{ann}(s_{\lambda})$. 
\end{example}

The following example, which is based on \cite[Example 9.2]{RT1}, shows that Proposition~\ref{prop:annschur} cannot be strengthened to the statement that $A/\operatorname{ann}(s_{\lambda})$ has the K\"{a}hler package. 

\begin{example}
Let $m = 3$ and $n = 2$ in 	Example~\ref{ex:denormalizedSchur}. Then 
$A = \mathbb{R}[x_1, x_2, x_3]/(x_1^3, x_2^3, x_3^3)$ 	has the K\"{a}hler package with respect to $\mathcal{K}_A = \R_{>0} x_1 + \R_{>0} x_2 + \R_{>0} x_3$.  Here $A$ is the cohomology ring of $\mathbb{P}^2 \times \mathbb{P}^2 \times \mathbb{P}^2$, and has Hilbert function $(1, 3, 6, 7, 6, 3, 1)$. 
We have $c_1 = x_1 + x_2 + x_3$, $c_2 = x_1x_2 + x_1x_3 + x_2x_3$, $c_3 = x_1x_2x_3$, and  $B = A[\zeta]/(\zeta^3 - c_1 \zeta^2 + c_2 \zeta - c_3)$ has the K\"{a}hler package with respect to $\mathcal{K}_B$.

Let $\lambda = (1, 1)$, so $s_{\lambda} = c_2$. Then Proposition~\ref{prop:annschur} implies that $\bar{A} = A/\operatorname{ann}(c_2)$ satisfies the mixed Hodge--Riemann relations in degree at most $1$ with respect to the image of $\mathcal{K}_A$. The Hilbert function of $\bar{A}$ is $(1, 3, 6, 3, 1)$. Using the basis $x_1x_2, x_1x_3, x_2x_3, x_1^2, x_2^2, x_3^2$ for $\bar{A}^2$, the Poincar\'{e} pairing on $\bar{A}^2$ is represented by the matrix
\[
\begin{pmatrix}
	0 & 1 & 1 & 0 & 0  & 1  \\ 
	1 & 0 & 1 & 0 & 1  & 0  \\ 
	1 & 1 & 0 & 1 & 0  & 0  \\ 
	0 & 0 & 1 & 0 & 0  & 0  \\ 
	0 & 1 & 0 & 0 & 0  & 0  \\ 
	1 & 0 & 0 & 0 & 0  & 0  \\ 
\end{pmatrix}.
\]
This matrix is nondegenerate and has signature $0$, i.e., it has $3$ positive eigenvalues and $3$ negative eigenvalues. If $\bar{A}$ satisfied the Hodge--Riemann relations, then this matrix would have 
signature $2$. 

Taking $h = x_1 + x_2 + x_3$, Theorem~\ref{thm:main} states that the form $(x, y) \mapsto \deg_A(xy c_2)$ is positive semidefinite on the subspace
$$\{a \in A^2 : a(h c_2 + s_{2,1}) = a(hs_{2,1} + s_{2,2} + s_{3,1}) = 0\}$$
of $A^2$. This subspace has dimension $2$, and the restriction of the form to this subspace is positive definite. 
\end{example}

\subsection{Strict positivity}

In this final section, we conjecture that a version of Theorem~\ref{thm:strict} holds in the setting of Theorem~\ref{thm:main}, giving a stronger version of the Hodge--Riemann property. We prove this conjecture in some special cases. We assume the setup of Theorem~\ref{thm:main} throughout.

\begin{conjecture}\label{q:strict}
	Assume the hypotheses of Theorem~\ref{thm:main}. 	Further assume that $\zeta$ satisfies the hard Lefschetz property in $B$ and that $\lambda$ has at most $r$ parts. 
	Then $a \neq 0$ if and only if $(-1)^k\deg_A(a^2 s_\lambda) > 0$.
\end{conjecture}

Using the equality statement in Theorem~\ref{thm:main}, this is 	equivalent to the following conjecture: if $a \neq 0$, then $a s_\mu \neq 0$ in $A$ for some partition $\mu$ with $\lambda \subseteq \mu$.
When $k = 0$, the conjecture follows from Theorem~\ref{thm:strict}. When $k = 1$, by Example~\ref{ex:rossToma}, the conjecture is that the bilinear form $A^1 \times A^1 \to \mathbb{R}$ given by $(x, y) \mapsto \deg_A(xys_{\lambda} )$ is nondegenerate when  $\zeta$ satisfies the hard Lefschetz property in $B$ and $\lambda$ has at most $r$ parts. Equivalently, multiplication by $s_{\lambda}$ defines an isomorphism from $A^1$ to $A^{n-1}$. 
In the geometric setting when $A, B$ arise from a smooth complex projective variety and an ample vector bundle $E$, this conjecture was proved in \cite{RT3}*{Theorem 1.1}. Furthermore, an extension of Conjecture~\ref{q:strict} to the setting of complex Lefschetz modules would imply \cite[Theorem 1.1]{HRLuZheng}. Below we give some further cases where Conjecture~\ref{q:strict} holds. 

\begin{example}
	Suppose that $\lambda = (n - 2k)$, so that $s_\lambda = s_{n - 2k}$. We claim that Conjecture~\ref{q:strict} holds in this case. 
	Indeed, suppose that $a$ is nonzero. Since $\zeta$ satisfies the hard Lefschetz property, $a \zeta^{n + (r - 1) - 2k} \in B^{n + (r - 1) - k}$ is nonzero.
	Recall from Section~\ref{sec:projbund} that  an element $y \in B$ is zero if and only if $\pi_*(y\zeta^i) = 0$ for all nonnegative integers $i$. Using \eqref{eq:pushforward}, we deduce that there is a nonnegative integer $i$ such that $\pi_*(a \zeta^{n + (r - 1) - 2k} \zeta^i)  = a s_{n - 2k + i} \neq 0$. 
\end{example}

\begin{example}\label{ex:strictci}
	Suppose that $\lambda = (1^{n - 2k})$, so $s_\lambda = c_{n - 2k}$. We claim that Conjecture~\ref{q:strict} holds in this case. 
	Indeed, suppose that $a s_\mu = 0$ for all $\lambda \subseteq\mu$. We want to show that $a = 0$. Recall that $n - 2k \le r$ by assumption. 
	We compute in $B$:
	\[
	0 = a (\zeta^r - c_1 \zeta^{r-1} + \dotsb + (-1)^r c_r) = 
	b \zeta^{r - (n - 2k - 1)} , 
	\]
	where $b = a (\zeta^{n - 2k - 1} - c_1 \zeta^{n - 2k - 2} + \dotsb + (-1)^{n - 2k - 1} c_{n - 2k - 1} ) \in B^{n - k - 1}$. 
	On the other hand, since $n - 2k - 1 \le r - 1$, it follows from \eqref{eq:Bdirectsum} that 
	$b$ is nonzero. 
	Since $\zeta$ satisfies the hard Lefschetz property, we have 
	\[
	0 \neq b \zeta^{n + (r - 1) - 2(n - k - 1)} = b \zeta^{r - (n - 2k - 1)},
	\]
	a contradiction. 
\end{example}

\begin{example}
	Suppose that $\lambda = (2,1^{n - 2k - 2})$. We claim that Conjecture~\ref{q:strict} holds in this case. 
	Indeed, let $m = \ell(\lambda) = n - 2k - 1 \le r$
	and let $C$ be the $m$-fold multiprojective bundle ring. 	Suppose that $a s_\mu = 0$ for all $\lambda \subseteq\mu$. We want to show that $a = 0$. 
	By Lemma~\ref{lem:detectzeroupstairs}, $a \Delta (\zeta_1 \cdots \zeta_m)^{r - m} s_\lambda(\zeta_1,\ldots,\zeta_m) = 0$ in $C$. 
		Let $\widehat{\lambda}_1 = (1^{n - 2k - 1})$. 
	By Example~\ref{ex:rectangle}, $s_\lambda(\zeta_1,\ldots,\zeta_m) =
	s_{\widehat{\lambda}_1}(\zeta_1,\ldots,\zeta_m)s_1(\zeta_1,\ldots,\zeta_m) = (\zeta_1 \cdots \zeta_m) (\zeta_1 + \cdots + \zeta_m)$. 
	Let $D = C/\ann((\zeta_1 \cdots \zeta_m)^{r - m + 1})$. Then 
	$a \Delta (\zeta_1 + \cdots + \zeta_m) = 0$ in $D$. Note that the top nonzero degree of $D$ is $2(k + m(m -1)/2) + 1$. 
	
	Recall that $\zeta$ satisfies the hard Lefschetz property in $B$ by assumption. After twisting by $-\epsilon h$ for some $0 < \epsilon \ll 1$ and $h \in \K_A$ (see Section~\ref{sec:strict}), Proposition~\ref{prop:convex} implies that $C$ has the K\"{a}hler package with respect to $p^*(\K_A) + \sum_{i = 1}^m \R_{> 0} (\zeta_i - \epsilon h)$. Observe that $\zeta_1 + \cdots + \zeta_m$ lies in the latter cone. By Proposition~\ref{prop:descent}, the image of $\zeta_1 + \cdots + \zeta_m$ satisfies the hard Lefschetz property in $D$. In particular, multiplication by $\zeta_1 + \cdots + \zeta_m$ is injective on $D^{k + m(m -1)/2}$. 
	Hence the fact that $a \Delta (\zeta_1 + \cdots + \zeta_m) = 0$ in $D$ implies that 
	$a \Delta = 0$ in $D$. Equivalently, $a \Delta (\zeta_1 \cdots \zeta_m)^{r - m} s_{\widehat{\lambda}_1}(\zeta_1,\ldots,\zeta_m) = 0$ in $C$. 
	 
	By Lemma~\ref{lem:detectzeroupstairs} and Example~\ref{ex:rectangle},	$a (s_{\widehat{\lambda}_1} \star_m s_\nu) = a s_{(1^m) + \nu} = 0$ in $A$ for all partitions $\nu$  with $\ell(\nu) \le m$. By Lemma~\ref{lem:idealgenerated},  $a s_\mu = 0$ for all partitions $\mu$ containing $(1^m)$, i.e., all partitions with at least $m$ parts. 
	We now argue as in Example~\ref{ex:strictci}. 	In $B$, we have
	\[
	0 = a (\zeta^r - c_1 \zeta^{r-1} + \dotsb + (-1)^r c_r) = 
	b \zeta^{r - (m - 1)} , 
	\]
	where $b = a (\zeta^{m - 1} - c_1 \zeta^{m - 2} + \dotsb + (-1)^{m - 1} c_{m - 1} ) \in B^{m + k - 1}$ is nonzero. In particular, $b \zeta^{r - m + 2} = 0$.	On the other hand, since $\zeta$ satisfies the hard Lefschetz property, we have 
	\[
	0 \neq b \zeta^{n + (r - 1) - 2(m + k - 1)} 
	= b \zeta^{r - m + 2},
	\]
	a contradiction.
\end{example}

\bibliography{matroid}
\bibliographystyle{amsalpha}

\end{document}